\documentclass[11pt]{amsart}
\usepackage{amsfonts}
\usepackage{amsthm,amscd,graphics,tabularx,array,CJK,fancyhdr}
\usepackage{amsmath}
\usepackage{amsfonts}
\usepackage{amsgen,xspace}
\usepackage{amssymb}
\usepackage{bbm}
\usepackage{booktabs}
\usepackage{cases}
\usepackage{color}
\usepackage{comment}
\usepackage{float}
\usepackage{geometry}
\usepackage{graphicx}
\usepackage{hyperref}
\usepackage{ifxetex}
\usepackage{mathrsfs}
\usepackage[section]{placeins}
\usepackage{threeparttable}
\usepackage{times}
\usepackage{txfonts}
\usepackage[arc,knot,all]{xy}

\newcommand{\R}{{\mathbb{R}}}
\newcommand{\mC}{{\mathbb{C}}}

\newcommand{\g}{{\mathfrak g}}

\newcommand{\h}{{\mathfrak h}}

\newcommand{\s}{{\mathfrak s}}
\newcommand{\fa}{{\mathfrak a}}

\newcommand{\fc}{{\mathfrak c}}

\newcommand{\ff}{{\mathfrak f}}
\newcommand{\fg}{{\mathfrak g}}
\newcommand{\fh}{{\mathfrak h}}

\newcommand{\fgl}{{\mathfrak g}{\mathfrak l}}
\newcommand{\fsl}{{\mathfrak s}{\mathfrak l}}

\newcommand{\fso}{{\mathfrak s}{\mathfrak o}}
\newcommand{\fsp}{{\mathfrak s}{\mathfrak p}}
\newcommand{\Sp}{{\operatorname{Sp}}}
\newcommand{\GL}{{\operatorname{GL}}}
\newcommand{\SL}{{\operatorname{SL}}}
\newcommand{\SO}{{\operatorname{SO}}}

\newcommand{\spin}{\operatorname{spin}}

\newcommand{\ad}{{\mathrm a}{\mathrm d\,}}

\newcommand{\diag}{{\mathrm d}{\mathrm i}{\mathrm a}{\mathrm g\,}}
\newcommand{\Span}{{\mathrm S}{\mathrm p}{\mathrm a}{\mathrm n\,}}

\newcommand{\End}{{\mathrm E}{\mathrm n}{\mathrm d\,}}

\newcommand{\Ker}{{\mathrm K}{\mathrm e}{\mathrm r\,}}

\newcommand{\rank}{{\mathrm r}{\mathrm a}{\mathrm n}{\mathrm k\,}}

\newcommand{\Max}{{\mathrm m}{\mathrm a}{\mathrm x\,}}
\newcommand{\tr}{{\mathrm T}{\mathrm r\,}}
\newcommand{\PV}{{ \mathbf P}{\mathbf V}}

\newcommand{\B}{{\operatorname{B}}}

\newcommand{\E}{{\operatorname{E}}}

\newcommand{\F}{{\operatorname{F}}}

 \makeatother

\newtheorem{thm}{Theorem}[section]
\newtheorem{defn}[thm]{Definition}

\newtheorem{prop}[thm]{Proposition}

\newtheorem{lem}[thm]{Lemma}
\newtheorem{rem}[thm]{Remark}
\newtheorem{cor}[thm]{Corollary}
\newtheorem{ex}[thm]{Example}

\begin{document}
\title[\'{E}tale prehomogeneous vector spaces]{Flat Lie groups, Frobenius Lie algebras and \'{e}tale prehomogeneous vector spaces for reductive Lie groups}
\author{Xiaomei Yang}
\address[Yang]{School of Mathematics, Sichuan Normal University, Chengdu, Sichuan 610068, P.~R.~China}
\email{xiaomeiyang@mail.nankai.edu.cn}
\author{Fuhai Zhu}
\address[Zhu]{Department of Mathematics, Nanjing University, Nanjing, P.~R.~China}
\email{zhufuhai@nju.edu.cn}
\date{}
\tableofcontents
\begin{abstract}
In this paper, we established the relationship among left-invariant flat connections on Lie groups, left-symmetric algebras, Frobenius Lie algebras and \'{e}tale prehomogeneous vector spaces, gave a one-to-one correspondence between the left-symmetric Lie algebras with a right identity and the \'{e}tale prehomogeneous vector spaces for a Lie group, and proved that, in essence, any left-symmetric structure on a reductive Lie algebra has a right identity, which implies that the classification of flat connections on a reductive Lie group $G$ amounts to that of \'{e}tale prehomogeneous vector spaces for $G$. We classified the \'{e}tale prehomogeneous vector spaces for $G$ with simple Levi factors.
\end{abstract}
\maketitle

\medskip
{\bf Keywords}: Left-invariant connection, Left-symmetric algebra, Frobenius Lie algebra, Prehomogeneous vector space, Symplectic Lie algebra.

\section{Introduction}\label{section1}

\subsection{Flat Lie Groups} Let $G$ be a Lie group with Lie algebra $\g$ and let $\Gamma(TG)$ denote the set of all
differentiable vector fields on $G$. A $G$-invariant affine connection $\nabla$
is called \textbf{torsion-free} 
if, for any $x,y\in\Gamma(TG)$,
\begin{equation}\label{torsionfree}
\nabla_xy-\nabla_yx-[x,y]=0,
\end{equation}
and  \textbf{flat}
if
\begin{equation}\label{flat}
\nabla_x\nabla_y-\nabla_y\nabla_x-\nabla_{[x,y]}=0.
\end{equation}
Lie groups with flat torsion-free left-invariant affine connections are called flat Lie groups.

\subsection{Flat Lie groups and Left-symmetric Algebras} Similar to the relation between Lie groups and Lie algebras, the flat connections on Lie groups correspond to special structures on there Lie algebras. To be precise, an affine connection on a Lie group $G$ determines a covariant differentiation:
$$\nabla_x:\Gamma(TG)\rightarrow\Gamma(TG),\ \  y\mapsto\nabla_xy, \ \ x,y\in\Gamma(TG).$$
We identify the Lie algebra $\g$ of $G$ with the set of left-invariant vector fields. Since $\nabla$ is $G$-invariant, for any $x,y\in\g$, $\nabla_xy$ is also in $\g$. If we put
\begin{equation*}
x\cdot y=xy=\nabla_xy,
\end{equation*}
then we obtain an $\R$-bilinear product on $\g$. The
vanishing of curvature and torsion, i.e., (\ref{torsionfree}) and
(\ref{flat}), are equivalent to the following identities:
\begin{eqnarray}
[x,y]&=&xy-yx,\nonumber\\
(xy)z-x(yz)&=& (yx)z-y(xz).\label{eq-lsa}
\end{eqnarray}
A vector space $\g$ over a field $\mathbb{F}$ with a bilinear
product $xy$ is called a \textbf{left-symmetric
algebra} (abbrev. LSA) if Equation~(\ref{eq-lsa}) holds. Equation~(\ref{eq-lsa}) implies that the commutators
\[[x, y]=x y-y x\]
satisfy the Jacobi identity. Accordingly, each left-symmetric
product has an adjacent Lie algebra. Note that the left-multiplication $L_x(y)=xy$ defines a Lie algebra representation:
$$L:\g\rightarrow\End(\g),\ \ \ x\mapsto L_x,$$
which is called the \textbf{left-regular representation} of $\g$.
It is easy to see that the classification of real/complex flat Lie groups is equivalent to that of real/complex left-symmetric algebras.

Left-symmetric algebras are a class of nonassociative algebras, which have close relation with many important fields in mathematics and mathematical physics. But their classification is very complicated due to the nonassociativity. During the past several decades, a great deal of effort has been devoted to classify left-symmetric algebras. Since there are no left-symmetric algebras on a finite-dimensional semisimple Lie algebra $\fg$ of characteristic $0$, many people studied the left-symmetric algebras for some low-dimensional nilpotent Lie algebras (see \cite{Bur1996-2,Kim1986,Sch1974}) or some special cases in certain higher dimensions (see \cite{Bur1998}). For reductive cases, Burde gave a classification of left-symmetric algebras for $\fgl(2)$ in \cite{Bur1996} and Baues investigated the left-symmetric structures for $\fgl(n)$ in \cite{Bau1999}.

\subsection{LSAs and Symplectic Lie Algebras} The classification of flat Lie groups and left-symmetric algebras is far from complete. Milnor initiated a systematical study for flat Lie groups with (pseudo-)Riemannian metric in \cite{Mil1976}, while B. Y. Chu \cite{Chu1974} found flat connections for Lie groups with left-invariant symplectic structures, which is the main ingredient in this paper.

If a Lie group $G$ is endowed with a left-invariant symplectic form $\omega$, then $\omega$ is also a nondegenerate closed $2$-form on $\fg$, in other words, $\omega$ is a nondegenerate skew-symmetric bilinear form with the following property:
\[d\omega(x,y,z)=\omega([x,y],z)+\omega([y,z],x)+\omega([z,x],y)=0.\]
Such Lie algebra $\fg$ is called a \textbf{symplectic Lie algebra}. Then
\[\omega(x, y\ast z)=\omega([x, y], z), \ \ \ \forall x, y,z\in\mathfrak{g}\]
defines a compatible left-symmetric product $\ast$ on $\mathfrak{g}$ (\cite[P.~154, Theorem 6]{Chu1974}.

For those Lie algebras without symplectic forms, we have the following result.

\begin{thm}[Theorem~\ref{thm-main1}]
  Let $(\g,\ast)$ be a left-symmetric algebra and let $(L^*,\g^*)$ be dual representation of the left-regular representation $(L,\g)$ of $\g$. Then $\g\dot+\g^*$ is a symplectic Lie algebra.
\end{thm}

As a corollary (Proposition~\ref{prop-perfect}) we give a simpler proof of the classical result of Helmstetter~\cite{Hel1971}: every perfect Lie algebra (i.e., $[\g,\g]=\g$) has no left-symmetric structure.

\subsection{Frobenius Lie Algebras and LSAs with Right Identities} For many Lie algebras, such as semisimple Lie algebras, the second cohomology with trivial coefficient is trivial. Therefore, the closed $2$-form $\omega$ is exact, i.e., $$\omega=df,\ \ \ \mbox{for some }f\in \fg^{\ast}.$$
A Lie algebra $\fg$ with a nondegenerate exact $2$-form $df$ is called \textbf{Frobenius}, which was introduced by Ooms (\cite{Oom1974,Oom1976,Oom1980}). Such a Lie algebra is interesting because whose index is zero. Here, by definition, the index of a Lie algebra is the minimum of the codimensions of coadjoint orbits. If the index of $\g$ is zero, then there is an open dense coadjoint orbit in $\g^*$. It is well-known that the index of a semisimple Lie algebra $\g$ is the rank of $\g$. Therefore, semisimple Lie algebras are not Frobenius, but many Parabolic subalgebras are Frobenius (see~\cite{DK2000,Ela1982,GG2008,Jos2006,Kos2012,TY2005} and reference therein). Then we have

\begin{thm}[Theorem~\ref{thm3.6}]
The left-symmetric algebra $(\g,\ast)$ has a right identity if
and only if $\fg\dot+\fg^{\ast}$ is a Frobenius Lie algebra with the natural symplectic form.
\end{thm}

Since every semisimple Lie algebra (which is perfect) has no left-symmetric structure, in this paper we will focus on reductive Lie algebras. With the help of symplectic forms, we have

\begin{thm}[Theorem~\ref{thm4.1}]
Let $\g$ be a reductive Lie algebra with a left-symmetric structure. Then $\g$ is decomposed into the direct sum of $2$ ideals: $\g=\g_1\oplus C$, where $C$ is an abelian ideal and $\g_1$ is a left-symmetric algebra with a right identity.
\end{thm}

\subsection{\'{E}tale Prehomogeneous Vector Spaces and LSAs with Right Identities} If a left-symmetric algebra $\g$ has a right identity $e$, then for the left-regular representation $(L,\g)$, the isotropy subalgebra at $e$ is trivial:
$$\g_e=\{x\in\g|L_xe=0\}.$$
Therefore $(\g,L,\g)$ is an \'{e}tale prehomogeneous vector space (see Definition~\ref{Def-PV}).

Around 1960, M. Sato introduced the notion of prehomogeneous vector spaces (abbrev. $\PV$s) and applied it to the study of Zeta function in \cite{SS1974}.
In the 1990s, Wright and Yukie showed that the theory of prehomogeneous vector spaces is
related to distributions of arithmetic objects(see \cite{WY1992}). As an application of the method, Kable and Yukie presented an analogous correspondence for some class of prehomogeneous vector spaces in \cite{KY1997,Yuk1997}.
Moreover, Yukie considered problems analogous to the Oppenheim conjecture from the viewpoint of prehomogeneous vector spaces in the series papers \cite{WYZ,Yuk1,Yuk3}.

A prehomogeneous vector space is defined as a finite-dimensional vector space $V$ together with a rational action of an algebraic group $G$ such that $G$ has an open dense orbit in $V$. We generalized the definition to representations of Lie groups and Lie algebras. As we mentioned above, the classification of flat connections on reductive Lie groups or left-symmetric structures on reductive Lie algebras amounts to that of \'{e}tale prehomogeneous vector spaces for reductive Lie groups or Lie algebras. For complex reductive Lie algebras with simple Levi factors we have the following

\begin{thm}[Theorems~\ref{thm-sc1},~\ref{s+c^{k}}]
Let $(\g_{s}\oplus \fgl(1)^{k},\sum_{i}\sigma_{i}\otimes\mu_{i},\sum_{i}V_{i}\otimes V'_{i})$ be an \'{e}tale $\PV$, where $\g_{s}$ is simple. Then it is equivalent to one of the following triplets:

 $(1) \ (\fsl(n)\oplus\fgl(1),\Lambda_{1}\otimes\mu_{1},V(n)\otimes V(n))\ (\tr(\mu_{1}(\fgl(1)))\not\equiv 0);$

$(2) \ (\fsl(2)\oplus\fgl(1),3\Lambda_{1}\otimes\Box,V(4));$

$(3) \ (\fsl(n)\oplus\fgl(1)^{n+1},\Lambda_{1}^{(\ast)}\otimes\mu_{1},V(n)\otimes V(n+1)),$ \ (for $\mu_{1},$ see Lemma~\ref{lemCij})$;$

$(4) \ (\fgl(n)\oplus\fgl(1)^{mn-n^2},\Lambda_{1}^{(\ast)}\otimes\mu_{1},V(n)\otimes V(m)),(m-n>1)$ \ (for $\mu_{1},$ see Lemma~\ref{abelian2})$;$

$(5) \ (\fsl(n)\oplus\fgl(1)^{n+1},\Lambda_{1}\otimes\mu_{1}+\Lambda_{1}^{\ast}\otimes\Box,V(n)+V(n)\otimes V(n)),$ \ (for $\mu_{1},$ see Lemma~\ref{SL11})$;$

$(6) \ (\fsl(2)\oplus\fgl(1)^{2},2\Lambda_1\otimes\Box+\Lambda_1\otimes\Box,V(3)+ V(2)).$

\end{thm}

The notation we used is standard. Let $\g$ be a complex simple Lie algebra of rank $n$ with a Cartan subalgebra $\fh$ and let $\Pi$ be the set of  simple roots of $\Delta(\g,\fh)$. Identify the real vector space spanned by $\Pi$ with $\R^n$. Let $\{\lambda_{1},\ldots,
\lambda_{n}\}$ be the standard orthonormal basis of $\R^n$. We write $\Lambda_k$ for the $k^{th}$ fundamental highest weight in the Bourbaki ordering of the simple roots.  For $\operatorname{A}_n = \fsl(n+1)$, $\Lambda_k = \lambda_1+\dots+\lambda_k$, but for the other classical
Lie algebras this is the case only for $k \leq n-2 \ (\operatorname{D}_n)$, $k \leq n-1
\ (\operatorname{B}_n, \ \operatorname{C}_n)$. For the exceptional simple Lie algebras it is a
little more complicated, but $\Lambda_1$ is the highest weight of
the representation of degree $26$ for $\operatorname{F}_4$, $27$ for $\operatorname{E}_6$, $56$ for $\operatorname{E}_7$ and $248$ for $\operatorname{E}_8$ and $\Lambda_2=\lambda_1+\lambda_2$ for $\operatorname{G}_2$. We denote by $V(n)$ the $n$-dimensional vector space over $\mC$ and by $M(n,m)$ the set of all $n\times m$ matrices. Sometimes we write $(\GL(n),m\Lambda_{k})$ instead of $(\GL(1)\times\SL(n),\Box\otimes m\Lambda_{k})$, where $\Box$ is the standard 1-dimensional representation of $\GL(1)$.


\subsection{Outline of The Paper} The paper is organized as follows. In Section~\ref{Relations} we established the relationship between left-symmetric algebras and symplectic Lie algebras, especially the relationship between Frobenius Lie algebras and left-symmetric algebras with a right identity, which was used in Section~\ref{lsa-reductive Lie algebra} to prove that any left-symmetric algebra with a reductive adjacent Lie algebra is a direct sum of an abelian ideal and an ideal with a right identity. In Section~\ref{Prehomogeneous}, we introduced the definition and some properties of prehomogeneous vector spaces, and proved that the classification of left-symmetric structures on reductive Lie algebras amounts to that of \'{e}tale prehomogeneous vector spaces for reductive Lie groups. Furthermore, we gave some results which provide some efficient and simple criteria to check the prehomogeneity for some complicated triplets, and classified \'{e}tale $\PV$s for reductive Lie groups with simple Levi factors in Sections~\ref{Cuspidal PV}, \ref{sec6}.

The relationship among flat Lie groups, left-symmetric algebras, symplectic Lie algebras, Frobenius Lie algebras and
prehomogeneous vector spaces may be illustrated
as follows:

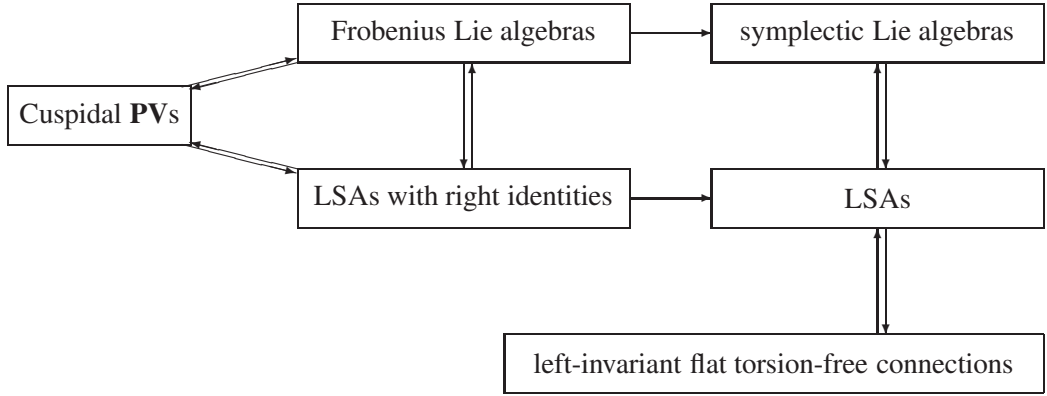
\begin{center}\setlength{\unitlength}{1.1mm}
\begin{picture}(120,60)
\put(-5,25){\framebox(25,15){$\begin{array}{cc}\text{Flat Lie groups}\end{array}$}}
\put(20,40){\vector(2,1){10}}
\put(30,45){\vector(-2,-1){10}}
\put(30,45){\framebox(20,10){\text{LSAs}}}
\put(20,25){\vector(2,-1){10}}
\put(30,20){\vector(-2,1){10}}
\put(40,45){\vector(0,-1){25}}
\put(40,20){\vector(0,1){25}}
\put(30,10){\framebox(20,10){$\begin{array}{cc}\text{Symplectic}\\ \text{Lie algebras}\end{array}$}}
\put(60,50){\vector(-1,0){10}}
\put(60,15){\vector(-1,0){10}}
\put(60,45){\framebox(20,10){$\begin{array}{cc}\text{LSAs with}\\ \text{right identities}\end{array}$}}
\put(70,45){\vector(0,-1){25}}
\put(70,25){\vector(0,1){20}}
\put(60,10){\framebox(20,10){$\begin{array}{cc}\text{Frobenius}\\ \text{Lie algebras}\end{array}$}}
\put(80,20){\vector(2,1){10}}
\put(90,25){\vector(-2,-1){10}}
\put(90,25){\framebox(20,15){\'{E}tale $\PV$s}}
\put(80,45){\vector(2,-1){10}}
\put(90,40){\vector(-2,1){10}}
\end{picture}
\end{center}

\section{Left-symmetric algebras, symplectic and Frobenius Lie algebras}\label{Relations}

\subsection{LSAs and symplectic Lie algebras}\label{section3.1}\

As mentioned in the introduction, the classification of flat Lie groups is equivalent to that of real left-symmetric algebras. By definition, a left-symmetric algebra is a vector space over a field $\mathbb{F}$ with a bilinear product $xy$ such that
$$(xy)z-x(yz)= (yx)z-y(xz).$$

In \cite{Chu1974}, B. Y. Chu introduced a method to construct LSAs from symplectic Lie algebras. To be precise, let $(\mathfrak{g}, \omega)$ be a symplectic Lie algebra. Then there exists a compatible left-symmetric product $\ast$ on
$\mathfrak{g}$ given by
\[\omega(x, y\ast z)=\omega([x, y], z), \ \ \ \forall x, y,z\in\mathfrak{g}.\]
It is well-known that every symplectic Lie algebra is even-dimensional, while there are odd-dimensional left-symmetric algebras, and an even-dimensional left-symmetric algebra may not be endowed with a symplectic form. The following result gives a method to construct symplectic Lie algebras from LSAs.

\begin{thm}\label{thm-main1}
  Let $(\g,\ast)$ be a left-symmetric algebra and let $(L^*,\g^*)$ be dual representation of the left-regular representation $(L,\g)$ of $\g$. Then there exists a natural Lie algebra structure on $\g\dot+\g^{\ast}$ such that $\g$ is a subalgebra, $\g^*$ is an abelian ideal, and
  \begin{equation}\label{eq-ad-L}
  [x,u]=L^*(x)(u)=-u\circ L_x,\ \ \ \forall x\in\g, u\in\g^*.
  \end{equation}
Furthermore, the natural skew-symmetric bilinear form on $\g\dot+\g^\ast$ defined by
$$\omega(x+u,y+v)=v(x)-u(y)$$
is a nondegerate closed $2$-form on $\g\dot+\g^*$, i.e., $\g\dot+\g^*$ is a symplectic Lie algebra.
  \end{thm}
\begin{proof} The first assertion is easy. We need to check that $\omega$ is closed.
Since $\omega(\fg,\fg)=\omega(\fg^{\ast},\fg^{\ast})=0$ and $\g^*$ is abelian, we need only consider
$$\omega([x,y],v)+\omega([y,v],x)+\omega([v,x],y)=0,$$
for any $x,y\in\g, v\in\g^{\ast}$, which
follows from the following:
\begin{eqnarray*}
\omega([x,y],v)+\omega([y,v],x)+\omega([v,x],y)
&=&v([x,y])-[y,v](x)-[v,x](y)\\
&=&v([x,y])+v(y\ast x)-v(x\ast y)\\
&=&0.
\end{eqnarray*}
Therefore, $\g\dot+\g^{\ast}$ is a symplectic Lie algebra.
\end{proof}

Now we give a new and simpler proof to the following well-known result of Helmstetter~\cite{Hel1971}.

\begin{prop}\label{prop-perfect}
Let $(\fg,\ast)$ be a left-symmetric algebra. Then $[\fg,\fg]\neq \fg.$
\end{prop}
\begin{proof}
Assume that $[\fg,\fg]=\fg$. Then there exist some $y_{i},z_{i}\in \fg$ such that $x=\sum\limits_{i}[y_{i},z_{i}]$ for any $x\in \fg$, it follows that
\[\tr(\ad_\g x)=\tr(\ad_\g \sum_{i}[y_{i},z_{i}])=0.\]
Consider the symplectic Lie algebra $\g\dot+\g^*$. By (\ref{eq-ad-L}), we have
\[\ad x(u)=L^{\ast}_{x}(u),\ \ \ \forall x\in\g, u\in\g^{\ast},\]
following which we have
\[\tr(\ad x|_{\g^{\ast}})=\tr (L^{\ast}_x)=\tr (L_{x})=\tr (L_{\sum\limits_{i}[y_{i},z_{i}]})=0.\]
Therefore, $\tr(\ad x)=0$. It is easy to see that $\tr(\ad u)=0$, for any $u\in\g^*$, which implies that $\g\dot+\g^*$ is a unimodular Lie algebra. By a result of Hano~\cite{Hano1957}, every unimodular symplectic Lie algebra is solvable. Therefore, $\g\dot+\g^*$ and $\g$ are solvable Lie algebras. A contradiction.
\end{proof}

\subsection{Frobenius Lie algebras and LSAs with right identities}\label{section3.2}\

Since Frobenius Lie algebras are special symplectic Lie algebras, it is natural to investigate the left-symmetric algebras constructed from Frobenius Lie algebras. Such left-symmetric algebras have a very special property.

\begin{prop}\label{prop3.5}
Let $(\g,\omega)$ be a symplectic Lie algebra. Then $(\fg,\omega)$ is a Frobenius Lie algebra if and only if the corresponding left-symmetric algebra has a right identity.
\end{prop}

\begin{proof}
Let $\fg$ be a Frobenius Lie algebra. Then
$\omega=df$ for some $f\in \fg^{\ast}$, which implies that
there exists a unique element $x_{f}\in \fg$
such that $\omega(x_{f},z)=f(z)$ for any $z\in \fg$. Hence we have
\[\omega(x,y)=f([x,y])=\omega(x_{f},[x,y])=-\omega(x,y\ast x_{f}),\]
for all $x,y\in \fg$. Therefore $y=-y\ast x_{f}$ and $-x_{f}$
is a right identity.

Now let $(\g,\ast)$ be a left-symmetric algebra with a right identity $e\in \fg$ and let $f(z)=\omega(z,e)$ for any $z\in \fg$. Then we have
\[\omega(x,y)=\omega(x,y\ast e)=\omega([x,y],e)=f([x,y])=df(x,y),\]
for any $x,y\in\g$. Thus $\omega=df$ and $\fg$ is a Frobenius Lie algebra.
\end{proof}

Combining the above Proposition with Theorem~\ref{thm-main1}, we have

\begin{thm}\label{thm3.6}
The left-symmetric algebra $(\g,\ast)$ has a right identity if
and only if $\fg\dot+\fg^{\ast}$ is a Frobenius Lie algebra.
\end{thm}

\begin{proof}
Let $\g$ be a left-symmetric algebra with a right identity $e$ and let $\omega$ be the natural symplectic form on $\g\dot+\g^*$. Then for any $x\in\g,v\in\g^*$, we have
$$\omega(x,v)=\omega(x*e,v)=\omega(x,v*e),$$
from which we deduce that $e$ is a right identity of $\g\dot+\g^*$. Thus $\g\dot+\g^*$ is a Frobenius Lie algebra.

Conversely, if $\g\dot+\g^*$ is a Frobenius Lie algebra, then, by Proposition~\ref{prop3.5}, we see that
$\g\dot+\g^{\ast}$ has a left-symmetric product with a right identity $e+v_0$, where $e\in\g,v_0\in\g^*$. For any $x\in\g$, we have
$$x=x*(e+v_0)=x*e+x*v_0,$$
following which we have $x*e=x$, i.e., $e$ is a right identity of $\g$.
\end{proof}

\section{Left-symmetric structures on reductive Lie algebras}\label{lsa-reductive Lie algebra}

From now on we focus on left-symmetric structures on reductive Lie algebras. In general, a left-symmetric algebra may not have a right identity. But, in essence, any left-symmetric structure on a reductive Lie algebra has a right identity. More precisely, we have 

\begin{thm}\label{thm4.1}
Let $\fg=\s\oplus \fc$ be a reductive Lie algebra with $\s=[\fg,\fg]$ and $\fc$ the center of $\fg$. If $\g$ has a compatible left-symmetric structure, then there exists a direct sum decomposition of $\g$ into left-symmetric ideals
$$\g=\g_1\oplus C,$$
where $C$ is an abelian left-symmetric algebra (hence a commutative associative algebra) and $\g_1$ has a unique right identity.
\end{thm}

For the proof of the above Theorem, the following result is crucial.

\begin{prop} \label{prop4.2.1}
Let $\fg=\s\oplus \fc$ be a reductive Lie algebra with a compatible left-symmetric structure $*$. Then there exist $\fg_{1}, \fc_1$ such that $\fg=\fg_{1}\oplus \fc_1$ is a direct sum of ideals of
the left-symmetric algebra $(\fg, \ast)$ with $\fg_{1}^{\s}=\{x\in \g_1|s*x=0, \forall s\in \s\}=\{0\}$.
\end{prop}

\begin{proof}
Let $\fg^{\ast}$ be the dual space of $\fg$. Then there exists a symplectic form $\omega$ on $\fg\dot+\g^\ast$
as defined in Theorem~\ref{thm-main1}. We still denote by $``\ast"$ the corresponding left-symmetric product
on $\g\dot+\g^\ast$. Naturally, we may write $\fg^{\ast}=V_{0}\dot+V_{1}$, where $V_{0}$ is the trivial
$\s$-module and $V_{1}$ is the direct sum of nontrivial $\s$-modules. Set $V_{0}^{\perp}=\fg_{1}\dot+V_{1}\dot+V_{0}$,
where $\fg_{1}=V_{0}^{\perp}\cap \fg$. Obviously $\s\subset \fg_{1}$, and there exists $\fc_1\subset \fc$ such that
$$(\fg_{1}\dot+V_{1})^{\perp}=(\fc_1\dot+V_{0}).$$

We will prove that both $(\fc_{1}\dot+V_{0})$ and $(\fg_{1}\dot+V_{1})$ are ideals of the left-symmetric algebra
$(\g\dot+\g^\ast,\ast)$ in the following steps.

(1) $\fg_{1}\ast V_{0}=V_{0}\ast \fg_{1}=\{0\}$.

For any $x\in \fg_{1}, \widetilde{y}\in \fg\dot+\g^\ast, v_{0}\in V_{0}$, we have
$$\omega(\widetilde{y}, x\ast v_{0})=\omega([\widetilde{y},x],v_{0})=0,$$
$$\omega(\widetilde{y}, v_{0}\ast x)=\omega([\widetilde{y},v_{0}],x)=0.$$
It follows that $\fg_{1}\ast V_{0}=V_{0}\ast \fg_{1}=\{0\}$.

(2) $\fg_{1}\ast \fc_1=\fc_1\ast \fg_{1}\subseteq \fc_1$.

For any $x\in \fg_{1}, v\in V_{1},c_{1}\in \fc_1$, from
$$\omega(v,x\ast c_{1})=\omega([v,x],c_{1})=0,$$
we get $\fg_{1}\ast c_{1}=c_{1}\ast \fg_{1}\subseteq \fc_{1}$.

(3) $V_{1}\ast \fc_1=\fc_1\ast V_{1}=\{0\}.$

For any $x\in \fg, v_{1}\in V_{1}, c_{1}\in \fc_1$, by
$$\omega(x,[v_{1},c_{1}])=-\omega(v_{1},[c_{1},x])-\omega(c_{1},[x,v_{1}])=0,$$
$$\omega(x,v_{1}\ast c_{1})=\omega([x,v_{1}],c_{1})=0,$$
 we obtain that $V_{1}\ast \fc_1=\fc_1\ast V_{1}=\{0\}.$

(4) $\fc_1\ast \fc_1\subseteq \fc_1$ and $\fc_1\ast V_{0}\subseteq V_{0}.$

For any $x\in V_{1}, y\in \fg_{1},c_{0},c_{1}\in \fc_1,v_{0}\in V_{0}$, using
$$\omega(x,c_{0}\ast c_{1})=\omega([x,c_{0}],c_{1})=0,$$
$$\omega(y,c_{0}\ast v_{0})=\omega([y,c_{0}],v_{0})=0,$$
we get $\fc_1\ast \fc_1\subseteq \fc_1$
and $\fc_1\ast V_{0}\subseteq V_{0}.$

Furthermore, for any $x\in\fg_{1}^{\s}$, we have $\omega([v,s],x)=\omega(v,s\ast x)=0$
for any $s\in\s$, $v\in V_1$, which implies that $x=0$ since $[\s,V_1]=V_1$ and $\omega$ is nondegenerate on $\g_1\dot+V_1$. Thus $\g_1^\s=\{0\}$.
\end{proof}

Now we need to consider the structure of $\g_1$, which is a reductive Lie algebra with a compatible left-symmetric structure such that $\g_1^\s=\{0\}$.

\begin{prop}\label{prop4.2.2}
Let $\fg_1=\s\oplus \fc$ be a reductive Lie algebra and let $(\rho,V)$ be
a representation of $\g$ with $\dim \g=\dim V$. Assume that
$V^\s=\{v\in V|\rho(\s)v=0\}=\{0\}$. If there exists
a sympletic structure $\omega$ on $\fg_1\dot+V^{\ast}$ with
$\omega|_{\fg_1\times \fg_1}=\omega|_{V^{\ast}\times V^{\ast}}=0$, then $\omega$ is an exact form, i.e.,
$\omega=df$ for some $f\in V$. Here $f$ is regarded as a linear function on $V^*$ and
$$df(x,v^*)=f(\rho(x)^*v^*)=v^*(\rho(x)f),\ \ \ \forall x\in\g_1,v^*\in V^*.$$
\end{prop}

\begin{proof}
Since $\omega|_{\fg_1\times \fg_1}=\omega|_{V^{\ast}\times V^{\ast}}=0$, we just need to consider
$\omega|_{\fg_1\times V^{\ast}}$. Note that  $\omega|_{\s\times V^{\ast}}$ may be regarded as an element in $B^{1}(\s,\rho,V)$. By the Whitehead Lemma (see Theorem 3.12.1 in \cite[P.~220]{Var1984}), there exists some $f\in V$ such that $\omega|_{\s\times V^{\ast}}(x,\cdot)=\rho(x)f$, for any $x\in\s$,
from which we deduce
\[\omega|_{\s\times V^{\ast}}(x,\nu^{\ast})=\rho(x)f(\nu^{\ast})=-f(\rho^*(x)\nu^*)=-df(x,\nu^{\ast}),\ \ \ \forall \nu^\ast\in V^*.\]
Furthermore, for any $s\in \s,c \in \fc, v\in V^{\ast}$, we have
\begin{eqnarray*}
\omega(c,[s,v])&=&-\omega(s,[v,c])-\omega(v,[c,s])=-\omega(s,[v,c])\\
&=&df(s,[v,c])=-df(v,[c,s])-df(c,[s,v])\\
&=&-df(c,[s,v]).
\end{eqnarray*}
It follows that $\omega(\fc,v)=-df(\fc,v)$, for any $v\in V$, since $\rho^{\ast}(\s)V^{\ast}=V^{\ast}$.
\end{proof}

Now we are in a position to prove Theorem~\ref{thm4.1}.

\begin{proof}[Proof of Theorem \ref{thm4.1}]
Combining the above Propositions~\ref{prop4.2.1},~\ref{prop4.2.2} with Proposition ~\ref{prop3.5}, we just need to show the uniqueness of the right identity of $\g_1$, which  follows from $\g_1^\s=\{0\}$ since the difference of any two right identities belongs to $\g_1^\s$. 
\end{proof}

\begin{rem}
The right identity of a left-symmetric algebra is not necessarily unique. Let $\g$ be the 2-dimensional Lie algebra with the standard basis $\{x,y\}$ such that $[x,y]=y$. Define a left-symmetric product on $\g$ by: $xx=-x$, $xy=0$, $yx=-y$ and $yy=0$. It is easy to check that $-x+ky$, for any $k\in\mC$, is a right identity.
\end{rem}

\section{Prehomogeneous vector spaces for Lie groups}\label{Prehomogeneous}

In order to classify the left-invariant flat connections on a reductive Lie group $G$, we need to classify left-symmetric structures on its Lie algebra $\g$, which, as we will see in this section, amounts to the classification of \'{e}tale prehomogeneous vector spaces for $G$ or $\g$.

\subsection{Preliminaries for prehomogeneous vector spaces}

The notion of prehomogeneous vector spaces was introduced by M. Sato around 1960 (see~\cite{Kim1998}). A prehomogeneous vector space is a triplet $(G,\rho,V)$ (or simply $(G,\rho)$) such that $(\rho,V)$ is a rational representation of an algebraic group $G$ and $G$ has an open dense orbit in $V$.
However, for the study of left-invariant flat connections on Lie groups, we need a more general definition of prehomogeneous vector spaces as follows.

\begin{defn}\label{Def-PV}
Let $G$ be a connected complex Lie group with Lie algebra $\g$ and $(\rho,V)$ a finite dimensional complex representation of $G$. We
call the triplet $(G,\rho,V)$ (or $(\fg,d\rho, V)$) a \textbf{prehomogeneous vector space}
(abbrev. $\PV$) if there exists $v\in V$ such that the orbit $G\cdot v$ is open in $V$ (or $d\rho(\fg)v=V$). Such point $v$ is called a \textbf{generic point}, and the isotropy subgroup $G_{v}=\{g\in G|\rho(g)v=v\}$ (or $\fg_{v}=\{x\in \fg|d\rho(x)v=0\}$) at a generic point $v$ is called a \textbf{generic isotropy subgroup} (or a \textbf{generic isotropy subalgebra}). 

A triplet $(G,\rho,V)$ is a $\PV$ only if $\dim G\geq \dim V$. In particular, a $\PV$ $(G,\rho,V)$ is called \textbf{\'{e}tale} if $\dim G=\dim V$. In this paper, we will focus on \'{e}tale $\PV$s.
\end{defn}


For Zariski topology, nonzero open subsets of irreducible varieties are always dense, which is false for real manifolds. Fortunately we have

\begin{prop}[{\cite[P.~618, Proposition 10.1]{Kna2002}}]
If $(G,\rho,V)$ is a prehomogeneous vector space for a complex Lie group $G$, then there is just one open orbit, and that orbit is dense.
\end{prop}

\begin{lem}[{\cite[P.~23, Proposition 2.2]{Kim1998}}]
Let $(\rho,V)$ be a complex representation of $G$, $v\in V$. Then the following conditions are equivalent:
\begin{itemize}
  \item [(1)] $\overline{\rho(G)v}=V;$
  \item [(2)] $\dim G_{v}=\dim G-\dim V;$
  \item [(3)] $\dim \g_{v}=\dim \g-\dim V;$
  \item [(4)] $\{d\rho(x)v|x\in \g\}=V$.
\end{itemize}
\end{lem}

\begin{ex} Let $(\rho,V)$ be a complex representation of $H$ with $\dim V\leq n$. Then the triplet
$(H\times \GL(n),\rho\otimes\Lambda_{1},V\otimes V(n))$ is always a $\PV$ and is called a \textbf{trivial $\PV$}.
\end{ex}


In the following we will collect some useful definitions and properties for $\PV$s and will see the difference between the cases for algebraic groups and those for Lie groups. Since we need the latter, we will modify the definitions and will give new proofs for each result when necessary.

\begin{defn}[{\cite[P.~245, Definition~7.39]{Kim1998}}]
Two triplets $(G_{i},\rho_{i},V_{i})$ $(i=1,2)$ are \textbf{equivalent} (or \textbf{isomorphic}) if there exist an isomorphism
$\sigma: \rho_{1}(G_{1})\rightarrow \rho_{2}(G_{2})$ of groups and an isomorphism $\tau: V_{1}\rightarrow V_{2}$
of vector spaces such that
$$\tau(\rho_{1}(g_{1})v_{1})=\sigma \rho_{1}(g_{1})(\tau(v_{1}))$$
for all $g_{1}\in G_{1}, v_{1}\in V_{1}$. That is to say, the diagram
\begin{displaymath}
\begin{array}{rcl}
 {V_{1}} & \stackrel{\tau}{\longrightarrow} & {V_{2}}\\
{\rho_{1}(g_{1})} \bigg{\downarrow} & {\curvearrowright} & \bigg{\downarrow}{\sigma\rho_{1}(g_{1})}  \\
{V_{1}}& \stackrel{\tau}{\longrightarrow}& {V_{2}}
\end{array}
\end{displaymath}
is commutative. Then we write $(G_{1},\rho_{1},V_{1})\cong (G_{2},\rho_{2},V_{2})$.
\end{defn}

By the above definition, we do not need to consider $G$ itself but the image $\rho(G)$.
For example, we have $(\SL(2)\times \SL(2),\Lambda_{1}\otimes\Lambda_{1},V(2)\otimes V(2))
\cong(\SO(4),\Lambda_{1},V(4))$, but $\SL(2)\times \SL(2)$ and $\SO(4)$ are not isomorphic. For the Lie algebra level, we may assume that the representations are faithful. Then $\g_i\cong d\rho_i(\g_i)$, which implies that $\g_1\cong\g_2$. Therefore, two $\PV$s $(\g_i,d\rho_i,V_i)$ are equivalent if and only if there exists a Lie algebra isomorphism $\sigma:\g_1\rightarrow\g_2$ such that $d\rho_1$ and $d\rho_2\circ\sigma$ are isomorphic as representations of $\g_1$. Furthermore, if $\g_1=\g_2$, then $\sigma$ is an automorphism of $\g_1$. In many cases, $\g_1$ has only inner automorphisms, which implies that $d\rho_1$ is isomorphic to $d\rho_2$.

The following Proposition is useful to classify $\PV$s for algebraic groups.

\begin{prop}[{\cite[P.~38, Proposition 2.21 and P.~245, Proposition 7.40]{Kim1998}}]\label{dual}
Let $(G,\rho,V)$ be a \textbf{reductive} $\PV$, i.e., $(\rho,V)$ is a rational representation for a reductive algebraic group $G$, and let $\rho^{\ast}:G\rightarrow \GL(V^{\ast})$ be the contragredient representation. Then the dual triplet $(G,\rho^{\ast},V^{\ast})$ is also a $\PV$. Moreover, $(G,\rho,V)\cong (G,\rho^{\ast},V^{\ast})$.
\end{prop}

Unfortunately, the above result does not hold for reductive Lie groups in general as the following example shows.

\begin{ex}
Let $G=\left\{\left.\left(\begin{array}{cccc}a_{1} & 0 & \cdots & 0 \\a_{2} & a_{1} & \cdots & 0 \\\vdots&\vdots&&\vdots\\a_{n} & 0 & \cdots & a_{1} \\\end{array}\right)\right|a_i\in\mC,a_1\neq 0\right\}$, $V=\mC^{n}, n>2$ and let $(\rho,V)$ be the natural representation of $G$.
It is easy to check that the triplet $(G,\rho,V)$ is a $\PV$. However,
\[\rho^{\ast}(g)\left(
           \begin{array}{c}
             y_{1} \\
             y_{2} \\
             \vdots \\
             y_{n} \\
           \end{array}
         \right)=\left(
\begin{array}{cccc}
\frac{1}{a_{1}} & -\frac{a_{2}}{a^{2}_{1}} & \cdots & -\frac{a_{n}}{a^{2}_{1}} \\
0 & \frac{1}{a_{1}} & \cdots & 0 \\
\vdots&\vdots&&\vdots\\
0 & 0 & \cdots & \frac{1}{a_{1}} \\
\end{array}
\right)\left(
           \begin{array}{c}
             y_{1} \\
             y_{2} \\
             \vdots \\
             y_{n} \\
           \end{array}
         \right)=\left(
           \begin{array}{c}
             \frac{1}{a_{1}}y_{1}-\sum_{i=2}^{n}\frac{a_{i}}{a^{2}_{1}}y_{i} \\
             \frac{1}{a_{1}}y_{2} \\
             \vdots \\
             \frac{1}{a_{1}}y_{n} \\
           \end{array}
         \right),\]
from which we deduce $(G,\rho^{\ast},V^{\ast})$ is not a $\PV$.
\end{ex}

The right statement should be the following, which can be easily proved.

\begin{prop}
Let $(G,\rho,V)$ be a $\PV$. If $\rho$ is completely reducible, then $(G,\rho^{\ast},V^{\ast})$ is a $\PV$.
\end{prop}

\begin{prop}\label{prop4.2.4}
Let $(G,\rho,V)$ be a $\PV$. If $ V_{0}$ is a submodule of $V$, then
$(G,\overline{\rho},\overline{V})$ is a $\PV$,
where $\overline{V}=V/V_{0}$ and $\overline{\rho}$ is the quotient representation.
\end{prop}

\begin{rem}
The submodule of a prehomogeneous vector space is not necessarily prehomogeneous. Consider the coadjoint representation of $2$-dimensional nonabelian Lie algebra with the standard basis $x,y$ such that $[x,y]=y$. Let $x^*,y^*$ be the dual basis of $\g^*$. Then $\langle x^*\rangle$ is a submodule, on which the restriction of the coadjoint action is trivial.
\end{rem}

The following result was crucial in the study of $\PV$s.

\begin{prop}[{\cite[P.~225, Theorem 7.3]{Kim1998}}]\label{castling}
Let $\rho: G\rightarrow \GL(V)$ be a finite-dimensional
representation and $\dim V=m>n\geq1$. Then the following statements are equivalent:
\begin{itemize}
  \item [(I)] $(G\times\GL(n),\rho\otimes\Lambda_{1},V\otimes V(n))$ is a
 $\PV;$
  \item [(II)] $(G\times\GL(m-n),\rho^{\ast}\otimes\Lambda_{1},V^{\ast}\otimes V(m-n))$
  is a $\PV$.
\end{itemize}
Furthermore, the generic isotropy subgroups of $(I)$ and $(II)$ are isomorphic.
\end{prop}
\begin{proof}
Choosing any basis of $V$ and $V(n)$, we may identify $V\otimes V(n)$ with ${M(m,n)}$. It is easy to see that any generic point $x$ in $V\otimes V(n)$ is identified with a matrix $X\in{M(m,n)}$ of maximal rank. Choosing suitable basis, one may identify the generic point $x$ with the matrix $\left(\begin{array}{c}I_n\\0\end{array}\right)$, where $I_n$ is the identity matrix. Identifying $G$ with $\rho(G)$ if necessary, we may regard $G$ as a subgroup of $\GL(m)$. Therefore, the isotropy subgroup of $x$ is

$$H_1=\left\{(Y,Z)\in G\times\GL(n)\left|Y=\left(\begin{array}{cc}A&B\\0&D\end{array}\right),Z=(A^{-1})^T\right.\right\}.$$

For the representation $\rho^\ast\otimes\Lambda_1$, the isotropy subgroup of $\left(\begin{array}{c}0\\I_{m-n}\end{array}\right)$ is
$$H_2=\left\{(Y,W)\in G\times\GL(m-n)\left|Y=\left(\begin{array}{cc}A&B\\0&D\end{array}\right),W=D^T\right.\right\}.$$
One may easily see that $H_1$ is isomorphic to $H_2$. Since
$$\dim H_1=\dim G+\dim\GL(n)-mn=\dim G-(m-n)n,$$
we have
$$\dim H_2=\dim G+\dim\GL(n-m)-m(m-n).$$
Thus $\left(\begin{array}{c}0\\I_{m-n}\end{array}\right)$ is a generic point in $V^\ast\otimes V(m-n)$.
\end{proof}

\begin{rem}
  From the proof, one may easily see that $\GL(n)$ and $\GL(m-n)$ in the statement may be replaced by $\SL(n)$ and $\SL(m-n)$, respectively.
\end{rem}

Naturally, we have the following useful definition.

\begin{defn}[{\cite[Definition~10]{SK1977}}]
We say that the prehomogeneous vector spaces $(I)$ and $(II)$ in Proposition~\ref{castling} are
\textbf{castling transforms} of each other. Two triplets $(G,\rho,V)$ and $(G',\rho',V')$ are said to be
\textbf{castling equivalent} if one is obtained from the other by a finite
number of successive castling transformations. A triplet $(G,\rho,V)$
is said to be \textbf{reduced} if $\dim V'\geq\dim V$ holds for any castling
transform $(G',\rho',V')$ of $(G,\rho,V)$.
\end{defn}

Our main purpose is to classify reduced $\PV$s.

\subsection{Some criteria for $\PV$s}

We now collect and prove some properties of $\PV$s , which were introduced to study $\PV$s for algebraic groups. We have to give these results new proofs since we are interested in the Lie groups case.

The following result is useful to construct new $\PV$s from a given one. In addition, those propositions built the relation between reductive Lie algebras with 2-simple Levi factors and simple Levi factors.

\begin{prop}[{\cite[P.~255, Proposition 7.52]{Kim1998}}]\label{isotropy subgroup}
Let $\rho_{i}:G\rightarrow \GL(V_{i})$ $(i=1,2)$ be finite-dimensional
representations. Then the following conditions are equivalent:
\begin{enumerate}
  \item [(I)]$(G,\rho_{1}\oplus\rho_{2},V_{1}\oplus V_{2})$ is a $\PV$.
  \item [(II)]$(G,\rho_{1},V_{1})$ is a $\PV$ and $(H,\rho_{2}|_{H},V_{2})$ is
             also a $\PV$, where $H$ denotes the generic
             isotropy subgroup of $(G,\rho_{1},V_{1})$.
\end{enumerate}
\end{prop}
\begin{proof}
(I)$\Rightarrow$ (II). Assume that $(G,\rho_{1}\oplus\rho_{2},V_{1}\oplus V_{2})$ is a $\PV$.
Then there exists an element $v=v_{1}+v_{2},v_{i}\in V_{i}$ such that
$d(\rho_{1}\oplus\rho_{2})(\fg)\cdot(v_{1}+v_{2})=d\rho_{1}(\fg)v_{1}\oplus d\rho_{2}(\fg)v_{2}=V_{1}\oplus V_{2}$, from which we deduce that $d\rho_{1}(\fg)v_{1}=V_{1}$ and $d\rho_{2}(\fg)v_{2}=V_{2}$. Thus we get the first part of the assertion. Since $\dim \fg-\dim(V_{1}\oplus V_{2})=\dim \fg_{v_{1}+v_{2}}$ and $\dim \fg-\dim V_{1}=\dim \fg_{v_{1}}$. We see that
$\dim \fg_{v_{1}+v_{2}} =\dim \fg_{v_{1}}-\dim  V_{2} $.

From
\begin{eqnarray*}
(\fg_{v_{1}})_{v_{2}}&=&\{x\in \fg_{v_{1}}| d\rho_{2}(x)(v_{2})=0\} \\
&=&\{x\in \fg| d\rho_{1}(x)(v_{1})=0\textrm{ and }d\rho_{2}(x)(v_{2})=0\} \\
&=&\fg_{v_{1}+v_{2}},
\end{eqnarray*}
we conclude that $\dim (\fg_{v_{1}})_{v_{2}}=\dim \fg_{v_{1}}-\dim V_{2}$.

(II)$\Rightarrow$ (I). Since $\dim \fg-\dim V_{1}=\dim \fg_{v_{1}}$ and
$\dim (\fg_{v_{1}})_{v_{2}}=\dim \fg_{v_{1}}-\dim V_{2}$. We have
$\dim \fg-\dim V_{1}-\dim V_{2}=\dim (\fg_{v_{1}})_{v_{2}}=\dim \fg_{v_{1}+v_{2}}$.
\end{proof}

\begin{lem}
  Let $(G,\rho,V)$ and $(G,\rho',V')$ be two triplets, and let $f:V\rightarrow V'$ be a $G$-equivariant continuous surjective map. If $(G,\rho,V)$ is a $\PV$, then so is $(G,\rho',V')$.
\end{lem}
\begin{proof}
  Since $f$ is $G$-equivariant, it maps an orbit in $V$ onto an orbit in $V'$. If $(G,\rho,V)$ is a $\PV$ and $O$ is the open dense orbit, then one can easily see that $f(O)$ is a dense orbit of $G$ in $V'$, and it must be open.
\end{proof}

Let $\rho:H\rightarrow\GL(V(d))$ be a representation of a Lie group $H$
on the $d$-dimensional vector space $V(d)$. Identify the vector space $\wedge^{2}(V(d))$ (resp. $S^{2}(V(d))$) with the set of all $d\times d$ skew-symmetric (resp. symmetric) matrices. Define the representation
$\wedge^{2}\rho$ (resp. $S^{2}\rho$) of $H$ on $\wedge^{2}(V(d))$ (resp. $S^{2}(V(d))$)
by $X\mapsto\rho(h)X\rho(h)^{T}$ for $X\in \wedge^{2}(V(d))$ (resp. $X\in S^{2}(V(d))$),
$h\in H$.

\begin{prop}[{\cite[P.~40, Proposition 13]{SK1977}}]\label{Sp(n)}
Assume that $2n\geq d$. Then  the following
statements are equivalent:
\begin{itemize}
  \item [(I)] $(\Sp(n)\times G,\Lambda_{1}\otimes\rho,V(2n)\otimes V(d))$ is a $\PV;$
  \item [(II)] $(G,\wedge^{2}\rho,\wedge^{2}(V(d)))$ is a $\PV$.
\end{itemize}

\end{prop}

\begin{proof}
Let $\langle,\rangle$ be a nondegenerate skew-symmetric bilinear form on $V(2n)\otimes V(2n)$
which is invariant under the action of $\Sp(n)$. For an element
$v=(v_{1},\ldots,v_{d})\in V(2n)\otimes V(d)$, let $f(v)$ be a $d\times d$ skew-symmetric
matrix with $(i,j)$-element $\langle v_{i},v_{j}\rangle, (i,j=1,\ldots,d)$. Then $f$ is surjective and $\Sp(n)\times G$-equivariant with $\Sp(n)$ acting on $\wedge^2(V(d))$ trivially. Therefore (I) implies (II).

For the converse statement, replacing $G$ by $\GL(d)$ and considering the two $\PV$s
$$(\Sp(n)\times \GL(d),\Lambda_{1}\otimes\rho, V(2n)\otimes V(d)),\ \ \ (\GL(d),\Lambda_{2},\wedge^2(V(d))),$$
one may easily get that $f:V(2n)\otimes V(d)\rightarrow\wedge^2(V(d))$ is surjective and $\Sp(n)\times\GL(d)$-equivariant.
Thus $f$ maps the open orbit in $V(2n)\otimes V(d)$ surjectively onto that in $\wedge^2(V(d))$, which implies that if $v\in V(2n)\otimes V(d)$ is a generic point, so is $f(v)\in\wedge^2(V(d))$. Let $H$ be the isotropy subgroup at $v$, and let $\pi:H\rightarrow\GL(d)$ be the natural homomorphism. Then $\Ker\pi$ is locally isomorphic to $\Sp(n-l)$ for $d=2l$ and to $\Sp(n-l-1)\ltimes U(2n-2l-1)$ for $d=2l+1$ (see Propositions 17 and 18 in Pages 102--104 \cite{SK1977}), and $\pi(H)<\GL(d)$ is the isotropy subgroup at $f(v)$.

If $v'\in\wedge^2(V(d))$ is a generic point for the action under $G$ with the isotropy subgroup $H'$, then $v'$ is a generic point for the $\GL(d)$-action and there is a generic point $v\in V(2n)\otimes V(d)$ such that $f(v)=v'$. Let $H_v$ be the isotropy subgroup at $v\in V(2n)\otimes V(d)$ for the $\Sp(n)\times G$-action. Then $\pi(H_v)\cong H'$ and $H_v$ is an extension of $H'$ over $\Ker\pi$. Noting that
$$\dim\Sp(n)+\dim G-\dim H_v=n(2n+1)-\dim\Ker\pi+(\dim G-\dim H')=2nd,$$
we have that $(\Sp(n)\times G,\Lambda_{1}\otimes\rho,V(2n)\otimes V(d))$ is a $\PV$.
\end{proof}

Replacing the nondegenerate skew-symmetric bilinear form on $V(2n)\times V(2n)$ by a symmetric one, one can similarly prove

\begin{prop}[{\cite[P.~41, Proposition 14]{SK1977}}]\label{SO(n)}
Assume that $n\geq d$. Then  the following statements are equivalent:
\begin{itemize}
  \item [(I)] $(\SO(n)\times G,\Lambda_{1}\otimes\rho,V(n)\otimes V(d))$ is a $\PV;$
  \item [(II)] $(G,S^{2}\rho,S^{2}(V(d)))$ is a $\PV$.
\end{itemize}

\end{prop}

By the same idea, we have
\begin{thm}[{\cite[Theorems~1.14~and~1.16]{KKTI1988}}]\label{saver}
Let $(\rho_{i},V(m_{i}))$ $(i=1,2)$ be finite-dimensional representations of $G$. If $n\geq \Max\{m_{1}, m_{2}\}$, then the following
statements are equivalent:
\begin{itemize}
  \item [(I)] $(G\times\GL(n)\ (\mbox{or\ $\SL(n)$}),\rho_{1}\otimes\Lambda_{1}+\rho_{2}\otimes\Lambda_{1}^{\ast},V(m_{1})\otimes V(n)+V(m_{2})\otimes V(n))$ is a $\PV;$
  \item [(II)] $(G,\rho_{1}\otimes\rho_{2},V(m_{1})\otimes V(m_{2}))$ is a $\PV$.
\end{itemize}
\end{thm}
\begin{proof}
Similar to the proof for Proposition \ref{Sp(n)}, we define a $G\times\GL(n)$-equivariant surjective map $f:V(m_{1})\otimes V(n)+V(m_{2})\otimes V(n)\rightarrow V(m_{1})\otimes V(m_{2})$ by $f(v)=f(X,Y)=XY^{T}$, where $\GL(n)$ acts on $V(m_1)\otimes V(m_2)$ trivially. Then (I) implies (II).

For the converse statement, the argument in \cite{KKTI1988} still works.
\end{proof}

By the above Theorem, one can get the following result which will be used later.
\begin{cor}\label{cor4.19}
If $m_1\geq 2$ and $m_2\geq 2$, then the triplet
$$(\fsl(n)\oplus\fgl(1)^k,\Lambda_{1}\otimes\mu_1+\Lambda_{1}^{\ast}\otimes \mu_2,V(n)\otimes V(m_1)\oplus V(n)\otimes V(m_2))$$
is not a $\PV$.
\end{cor}

\subsection{\'{E}tale PVs and LSAs}\label{section3.3}

For Frobenius Lie algebras, the adjoint representations are $\PV$s. More generally, we have

\begin{thm}\label{thm3.2.1}
Let $(\fg,\ast)$ be a left-symmetric algebra with a right identity. Then the left regular representation gives an \'{e}tale $\PV$. Conversely, if $(\fg,\rho,V)$ is an \'{e}tale $\PV$ for a Lie algebra $\fg$, then $\fg$ has a left-symmetric structure with a right identity.
\end{thm}

\begin{proof}
Let $e$ be a right identity of a left-symmetric algebra $(\fg,\ast)$ and let
$L:\fg\rightarrow \End(\fg)$ be the left regular representation. It is easy to see that $(\fg, L,\fg)$ is an \'{e}tale $\PV$ since $e$ is a generic point.

Conversely, let $(\g,\rho,V)$ be an \'{e}tale $\PV$. Then there exists an element $v\in V$ such that $\{\rho(x) v \mid x\in \fg\}=V$ and the mapping $x\mapsto \rho(x)v$ is a bijection from $\g$ onto $V$. Therefore, for any $x,y\in\g$, there exists a unique $z\in\g$ such that
\begin{equation}\label{eq_pv2lsa}
\rho(x)\rho(y)v=\rho(z)v,
\end{equation}
which defines a product ``$*$'' on $\g$ by
$x*y=z.$
One may easily show that $(\g,*)$ is a left-symmetric algebra. Furthermore, there exists a unique $e\in\g$ such that $\rho(e)v=v$. Equation~$(\ref{eq_pv2lsa})$ shows that $x*e=x$, i.e., $e$ is a right identity.
\end{proof}

\begin{prop}
Let the two triplets $(\fg,\rho_{1},V_{1})$ and $(\fg,\rho_{2},V_{2})$ be
equivalent \'{e}tale $\PV$s. Then $(\fg,\ast)$ and $(\fg,\cdot)$
are isomorphic left-symmetric algebras, where ``$\ast$'' and ``$\cdot$''
are the corresponding left-symmetric structures on $\fg$ respectively.
\end{prop}
\begin{proof}
Since $(\fg,\rho_{1},V_{1})\cong(\fg,\rho_{2},V_{2})$ are equivalent, there exists an isomorphism
$\sigma: \rho_{1}(\fg)\rightarrow \rho_{2}(\fg)$ of Lie algebras and an isomorphism
$\tau: V_{1}\rightarrow V_{2}$ of vector spaces such that
$$\tau(\rho_{1}(x)v_{1})=\sigma \rho_{1}(x)(\tau(v_{1}))$$
for all $x\in \fg, v_{1}\in V_{1}$.

Since $(\fg,\rho_{1},V_{1})$ and $(\fg,\rho_{2},V_{2})$ are \'{e}tale $\PV$s, we see that the isomorphism $\sigma$ induces an automorphism $h$ of $\fg$ such that the following diagram commutes:
\begin{displaymath}
\begin{array}{rcl}
 {\fg} & \stackrel{h}{\longrightarrow} & {\fg}\\
\bigg{\downarrow} & {\curvearrowright} & \bigg{\downarrow}  \\
{\rho_{1}(\fg)}& \stackrel{\sigma}{\longrightarrow}& {\rho_{2}(\fg)}
\end{array}
\end{displaymath}
By Theorem \ref{thm3.2.1}, we know that there are two products
``$\ast$'' and ``$\cdot$'' on
$\fg$ such that $\rho_{1}(x)\rho_{1}(y)=\rho_{1}(x\ast y)$ and
$\rho_{2}(x)\rho_{2}(y)=\rho_{2}(x\cdot y)$. Furthermore, we have
$$\sigma\rho_{1}(x)\sigma\rho_{1}(y)\tau(v_{1})
=\sigma\rho_{1}(x)(\tau(\rho_{1}(y)v_{1}))
=\tau(\rho_{1}(x)\rho_{1}(y)v_{1})
=\tau(\rho_{1}(x\ast y)v_{1}),$$
and
\begin{eqnarray*}
\sigma\rho_{1}(x)\sigma\rho_{1}(y)\tau(v_{1})&=&\rho_{2}(h(x))\rho_{2}(h(y))\tau(v_{1})\\
&=&\rho_{2}(h(x)\cdot h(y))\tau(v_{1})\\
&=&\sigma\rho_{1}(h^{-1}(h(x)\cdot h(y)))\tau(v_{1}).
\end{eqnarray*}
Thus we obtain $x\ast y=h^{-1}(h(x)\cdot h(y))$ for any $x,y\in\fg$.
\end{proof}
\begin{prop}
Let $(\fg_{1},\ast)$ and $(\fg_{2},\cdot)$ be two isomorphic left-symmetric algebras with a right identity and let $(\fg_1, L_{{1}},\g_1)$ and $(\fg_2,L_2,\g_2)$ be
the corresponding left-regular representations. Then the two $\PV$s
$(\fg_1, L_{{1}},\fg_{1})$ and $(\fg_2,L_{{2}},\g_{2})$ are equivalent.
\end{prop}

\begin{proof}
Let $\varphi:\g_1\rightarrow\g_2$ be an isomorphism of left-symmetric algebras. Define $\sigma: L_{1}(\fg)\rightarrow L_{2}(\fg)$ by
\[\sigma L_{1}(x)=L_{2}(\varphi(x)).\]
One can easily see that $\sigma$ is well-defined and is an isomorphism of Lie algebras.
It follows that $\varphi(L_1(x)(y))=\sigma L_{{2}}(x)(\varphi(y))$ and the assertion follows.
\end{proof}

\section{\'{E}tale PVs for reductive Lie algebras with simple Levi factors I}\label{Cuspidal PV}

Now we want to classify \'{e}tale $\PV$s for reductive Lie algebras. In this section we consider reductive Lie algebras with simple Levi factors and 1-dimensional center. The main result is known, see, for example, \cite{Bau1999,Bur1996}. The main ingredient in their papers are rational representations and $h$-transformations (see \cite{Hel1971}). As we mentioned before, left-regular representations are not necessarily rational, even for reductive Lie algebras. We applied different method, which depends on Theorem~\ref{thm4.1}, saying that left-symmetric algebras $\g$ with reductive adjacent Lie algebras must be the direct sum of two ideals, one of which is abelian and the other has a right identity. Therefore, if the reductive Lie algebra $\g$ has 1-dimensional center, then it is easy to see that $\g$ is simple as left-symmetric algebra, hence it has a right identity, from which we deduce that the left symmetric structures on $\g$ are 1-1 correspondence with the \'{e}tale $\PV$s for $\g$.

We will see that if $\g$ has 1-dimensional center, then the classification of \'{e}tale $\PV$s amounts to that of rational ones. The key observation is the following. Let $\g=\s\oplus\fc$ be the direct sum of two ideals with $\fc$ one-dimensional and
let $(\g,\rho,V)$ be an \'{e}tale $\PV$.
For any $c\in \fc$, let $\rho(c)=c_s+c_n$ be the Jordan-Chevalley decomposition of $\rho(c)$. Define $\pi:\g\rightarrow \fgl(V)$ by $\pi(s+c)=\rho(s)+c_s$. It is easy to check that $\pi$ is a representation of $\g$. Then the following observation is crucial for the discussion in the section.

\begin{lem}\label{lem-semisimple}
Let $\g=\s\oplus\fc$ with $\s$ a perfect Lie algebra, i.e., $[\s,\s]=\s$. Then the triplet $(\g,\rho,V)$ is a $\PV$ if and only if so is $(\g,\pi,V)$.
\end{lem}
\begin{proof}
Let the notation be as above. For nonzero $c\in\fc$, let $\rho(c)=c_s+c_n$ be the Jordan-Chevalley decompositon of $\rho(c)$. Let $\tilde{\g}=\g\oplus \mC c_{1}$ be a direct sum of $\g$ and a one-dimensional ideal $\mC c_{1}$. Define $\tilde{\rho}(x+kc_{1})=\rho(x)+kc_n$, for any $x\in\g$, $k\in\mC$. One can easily show that $\tilde{\rho}$ is a representation of $\tilde{\g}$. For a generic point $v\in V$, i.e., $\rho(\g)v=V$, $\tilde{\g}_v=\{x\in\tilde{\g}|\tilde{\rho}(x)v=0\}$ must be one-dimensional with a basis $x_{0}+c_{1}$ for some $x_{0}\in\g$.

Let $\g'=\s\oplus\mC c_{1}$. Then we will show that $x_{0}+c_{1}\in\g'$, i.e., $x_{0}\in\s$.

It is enough to show that the triplet $(\g',\pi,V)$ is not a $\PV$. Otherwise, $\g'\dot +V^*$ is a Frobenius Lie algebra. Noticing that $\s$ is unimodular, and $\tr\pi(c_{1})=0$, we have that $\g'\dot+V^*$ is a unimodular Frobenius Lie algebra and it must be solvable, which contradicts the fact that $\s$ is perfect. Since $\tilde{\g}_v$ is one-dimensional and $\dim\g'=\dim\tilde{\g}-1$, one deduces that $\tilde{\g}_v\subseteq\g'$.

Now consider the subalgebra $\ff=\s\oplus \mC(x_0-c_{1})$. Since $\ff\dot+\tilde{\g}_v=\tilde{\g}$, the triplet $(\ff,\pi,V)$ is a $\PV$. It is easy to see that the triplet $(\ff,\pi,V)$ is isomorphic to $(\g,\pi,V)$.

The proof of the converse statement is easy.
\end{proof}

The above Lemma is useful as we will see in the following example. The same result can be found in \cite[P.~13, Lemma 10]{Bur1996}, but our argument is simpler.

\begin{ex} Let $\g$ be a left-symmetric algebra with $\fgl(2)$ as the adjacent Lie algebra. We may assume that $L_c$ is semisimple, where $c=I_2$ is in the center of $\g$. Since the left-regular representation is completely reducible as an $\fsl(2)$-module, it is isomorphic to either $V(4)$ or $V(2)\otimes V(2)$. Since $L_c$ commutes with the action of $\fsl(2)$, in the first case, $L_c$ is a scalar, which we may assume to be 1. In the second case, identifying $V(2)\otimes V(2)$ with $M(2,2)$, the action of $\fsl(2)$ is the left-multiplication and $L_c$ is a right-multiplication by a diagonal $2\times 2$ matrix $A$ with nonzero trace. Up to a scalar, we have $A$ has one of the forms
 \[(1)\ I_4,\ \ \ (2)\ \left(
    \begin{array}{cc}
      1 & 0 \\
      0 & \lambda \\
    \end{array}
  \right)\textrm{ with $|\lambda|\leq 1$ and }\lambda\neq -1.\]
Thus one can easily check that the \'{e}tale $\PV$ is equivalent to one of the following triplets:

(1) $(\fsl(2)\oplus\fgl(1),3\Lambda_{1}\otimes\Box,V(4))$;

(2) $(\fsl(2)\oplus\fgl(1),\Lambda_{1}\otimes\tau,V(2)\otimes V(2))$, $\tau(c)=\diag(1,\lambda)$.

The non-completely reducible cases are deformations of the above and we just need to find some nonzero nilpotent matrix which commutes with $L_c$ (and the action of $\fsl(2)$). It is easy to see that we have a third case:
$$(3)\ \left(
    \begin{array}{cc}
      1 & 1 \\
      0 & 1 \\
    \end{array}
  \right),$$
which is a deformation of case (2) when $\lambda=1$. Note that when $\lambda=1$ in case (2), the left-symmetric product is the usual matrix multiplication.

\end{ex}

The above example can be easily generalized. For $\fgl(n)=\fgl(1)\oplus\fsl(n)$, define a triplet
$$(\fsl(n)\oplus\fgl(1),\Lambda_{1}\otimes\tau,V(n)\otimes V(n))\ (\tr \tau(\fgl(1))\not\equiv 0).$$
It is easy to check this triplet is an \'{e}tale $\PV$. Furthermore, we have

\begin{thm} \label{thm-sc1}
Let $\g=\g_s\oplus\fgl(1)$ and let $(\g,\rho,V)$ be an \'{e}tale $\PV$.
Then $(\g,\rho,V)$ is equivalent to one of the following triplets:

(1) $(\fsl(n)\oplus\fgl(1),\Lambda_{1}^{(*)}\otimes\tau,V(n)\otimes V(n))$, where $\tau(c)=J$ and $J$ is a Jordan canonical form with $\tr(J)\neq 0$;

(2) $(\fsl(2)\oplus\fgl(1),3\Lambda_{1}\otimes\Box,V(4))$.
\end{thm}

The proof of the Theorem will occupy the rest of this section. Thanks to Lemma~\ref{lem-semisimple}, we just need to determine all completely reducible \'{e}tale $\PV$s. It is easy to see that those $\PV$s are direct sum of irreducible $\PV$s, which are classified by Sato and Kimura \cite{SK1977}.

\begin{prop}[{\cite[\S 7]{SK1977}}]\label{prop6.7}
Let $(\g_{s}\oplus \fgl(1),\rho,V)$ be a nontrivial irreducible $\PV$, where $\g_s$ is simple and $\fgl(1)$ is $1$-dimensional center, and let $\fh$ be the generic isotropy subalgebra. Then it is equivalent to one of the following triplets, where $\Lambda^{(\ast)}$ means $\Lambda$ or $\Lambda^*:$

1. Regular $\PV$s.
\begin{enumerate}
\item $(\fgl(1),\Lambda_{1},V(1))$, $\fh=\{0\};$
\item $(\fgl(n),2\Lambda_{1}^{(\ast)},V(n(n+1)/2)),(n \geq 2),\ \fh\cong\fso(n);$
\item $(\fgl(2m),\Lambda_{2}^{(\ast)},V(m(2m-1))),(m\geq 3),\ \fh\cong \fsp(m);$
\item $(\fgl(2),3\Lambda_{1},V(4)),\ \fh\cong \{0\};$
\item $(\fgl(6),\Lambda_{3},V(20)),\ \fh\cong\fsl(3)\oplus \fsl(3);$
\item $(\fgl(7),\Lambda_{3}^{(\ast)},V(35)),\ \fh\cong \operatorname{G}_2;$
\item $(\fgl(8),\Lambda_{3}^{(\ast)},V(56)),\ \fh\cong \fsl(3);$
\item $(\fso(7)\oplus\fgl(1),\spin\otimes\Box,V(8)),\ \fh\cong \operatorname{G}_2;$
\item $(\fso(9)\oplus\fgl(1),\spin\otimes\Box,V(16)),\ \fh\cong \fso(7);$
\item $(\fso(11)\oplus\fgl(1),\spin\otimes\Box,V(32)),\ \fh\cong \fsl(5);$
\item $(\fso(12)\oplus\fgl(1),\spin^{\pm}\otimes\Box,V(32)),\ \fh\cong \fsl(6);$
\item $(\fso(14)\oplus\fgl(1),\spin^{\pm}\otimes\Box,V(64)),\ \fh\cong \operatorname{G}_2\oplus \operatorname{G}_2;$
\item $(\fsp(3)\oplus\fgl(1),\Lambda_{3}\otimes\Box,V(14)),\ \fh\cong \fsl(3);$
\item $( \operatorname{G}_2\oplus\fgl(1),\Lambda_2\otimes\Box,V(7)),\ \fh\cong \fsl(3);$
\item $( \operatorname{E}_{6}\oplus\fgl(1),\Lambda_1\otimes\Box,V(27)),\ \fh\cong \operatorname{F}_4;$
\item $( \operatorname{E}_{7}\oplus\fgl(1),\Lambda_1\otimes\Box,V(56)),\ \fh\cong \operatorname{E}_6.$
\end{enumerate}

2. Nonregular $\PV$s.
\begin{enumerate}
\item $(\fgl(n),\Lambda_{1}^{(\ast)},V(n));$
\item  $(\fgl(2m+1),\Lambda_{2}^{(\ast)},V(m(2m+1))),(m\geq 2);$
\item $(\fso(n)\oplus\fgl(1),\Lambda_{1}\otimes\Box,V(n))$, $\fh\cong\fso(n-1)\oplus\fgl(1);$
\item $(\fso(10)\oplus\fgl(1),\spin^{\pm}\otimes\Box,V(16));$
\item $(\fsp(n)\oplus\fgl(1),\Lambda_{1}\otimes\Box,V(2n))$ $(n\geq2);$
\end{enumerate}
\end{prop}

Let $(\g,\rho,V)$ be a completely reducible  \'{e}tale $\PV$, say $V=V_1\oplus V_2$, where $V_1$ is irreducible. Then $(\g,\rho_1,V_1)$ is one of the triplets in the above Proposition. Therefore $(\fh,\rho_2,V_2)$ is an \'{e}tale $\PV$ by Proposition~\ref{isotropy subgroup}, following which $\fh$ is not semisimple. Thus we only need to consider the $5$ nonregular cases.
Now we deal with them case by case and the condition $\dim V=\dim\g$ is crucial.

\textbf{I. Case (5): $\g\cong\fsp(n)\oplus \fgl(1)$.}

For $n\geq 3$, $V$ is the direct sum of $k$ copies of $V(2n)$, and we may identify $V$ with $V(2n)\otimes V(k)$. Then $\dim V=2kn$. Since $\dim\g=n(2n+1)+1$, it is impossible to have $\dim V=\dim \g$.

For $n=2$, since $\fsp(2)\cong\fso(5)$, $V$ may be the direct sum of some copies of $V(4)$ and $V(5)$, which is also impossible since $\dim\g=11$.

\textbf{II. Cases (3) and (4): $\g\cong \fso(n)\oplus \fgl(1)$, $n>5$.}

If $n\neq 10$, then $V$ is the direct sum of $k$ copies of $V(n)$. Similar argument as case I shows that it is impossible.

For $n=10$, we have $\dim\g=46$, thus $V$ must be the direct sum of $3$ copies of $V(10)$ and $1$ copy of $V(16)$. Easy calculation shows that $(\fgl(1),S^{2}\mu_{1},V(6))$ is not a $\PV$.


\textbf{III. Cases (1) and (2): $\g\cong\fgl(n)$, $n\geq 3$.}

The direct summand of $(\rho,V)$ must be $(\Lambda_1^{(*)},V(n))$ or $(\Lambda_2^{(*)},V(\frac{m}{2}(m+1)))$ $(n=2m+1)$. We just need to show that $\Lambda_1+\Lambda_1^*,\Lambda_2+\Lambda_1^{(*)}$ do not occur in $\rho$, which is true thanks to the following Lemmas.

\begin{lem}\label{prop4.3.6}
The triplet $(\fgl(n), \Lambda_{1}\oplus \Lambda_{1}^{\ast}$, $V(n)\oplus V(n))$ $(n\geq 3)$
is not a $\PV$.
\end{lem}

\begin{proof}
Let $X_{0}=(1,0,\ldots,0)^{T}$. Then the generic isotropy subalgebra of $(\fgl(n),\Lambda_{1},V(n))$ is
$\g_{X_{0}}=(0,A)$, where $A\in M(n,n-1)$. But for any $v\in V(n)$, $\dim \Lambda_1^*(\g_{X_{0}})v\leq n-1$,
hence $(\g_{X_{0}}, \Lambda_{1}^{\ast},V(n))$ $(n\geq 3)$ is not a $\PV$.
\end{proof}

\begin{rem}
From the proof one can easily see that $\Lambda_1\oplus\Lambda_1^*$ may be $\fsl(n)\oplus\fgl(1)^2$-prehomogeneous for suitable $\fgl(1)^2$ action, which will be used in the last section.
\end{rem}

Note that if $\g$ has $2$-dimensional center, the result will be different:

\begin{rem} The triplet $(\fsl(n)\oplus\fgl(1)^2,\Lambda_{1}\oplus \Lambda_{1}^{\ast},V(n)\oplus V(n))$ $(n\geq 3)$
is a $\PV$. Here $\fgl(1)^2$ acts on each summand with natural multiplication.
\end{rem}

The above observation will be used in the next section when we deal with reductive Lie algebras with multi-dimensional center.

\begin{lem}\label{lem6.171}
The triplet
$(\fgl(2m+1), \Lambda_{2}\oplus\Lambda_{1},V(m(2m+1))\oplus V(2m+1))\ (m\geq 2)$
is a $\PV$ and its generic isotropy subalgebra is $\fsp(m)$.
\end{lem}
\begin{proof}
We identify $V=V(m(2m+1))\oplus V(2m+1)$ with $(X_{1},X_{2})\in \{M(2m+1,2m+1)\oplus M(2m+1,1)|X_{1}^{T}=-X_{1}\}$. Then
$\rho(A,C)X=(AX_{1}+X_{1}A^{T},AX_{2})$, where $X=(X_{1},X_{2})\in V$ and
$A\in \fgl(2m+1)$.

Take
 \[X_{0}=\left(\left(
                  \begin{array}{ccc}
                    0 &  I_{m} &0\\
                    -I_{m} & 0 & 0 \\
                    0& 0 & 0 \\
                  \end{array}
                \right)
,(0,0,\ldots,0,1)^{T}
 \right).\]
Easy calculation shows that the isotropy subalgebra $\fg_{X_{0}}$ is $\fsp(m)$.
Since $\dim \g_{X_{0}}=m(2m+1)=\dim\fgl(2m+1)-\dim V$, it is a $\PV$.
\end{proof}

\begin{lem}\label{lem5.9}
The triplet $(\fgl(2m+1), \Lambda_{2}\oplus\Lambda_{1}^{\ast},V(m(2m+1))\oplus V(2m+1))\ (m\geq 2)$
is not a $\PV$.
\end{lem}
\begin{proof}
The generic isotropy subalgebra of $(\fgl(2m+1), \Lambda_{2},V(m(2m+1)))$ is
\[\fh=\left\{\left(
      \begin{array}{cc}
        A & \beta \\
        0 & c \\
      \end{array}
    \right)\in\fgl(2m+1), A\in\fsp(m)\right\}.\]
Let $X_{0}=\left(
             \begin{array}{c}
               X_{1} \\
               X_{2} \\
             \end{array}
           \right)\in M(2m+1,1)$, where $X_1\in M(2m,1)$. Then the action of $\fh$ on $X_{0}$ is given by
\[-\left(
    \begin{array}{cc}
      A^{T} & 0 \\
      \beta^{T} & c \\
    \end{array}
  \right)\left(
           \begin{array}{c}
             X_{1} \\
             X_{2} \\
           \end{array}
         \right)=-\left(
                   \begin{array}{c}
                     A^{T}X_{1} \\
                     \beta^{T}X_{1}+cX_{2} \\
                   \end{array}
                 \right)
.\]
Hence the action of $\fh$ on $X_{1}$ is $(\fsp(m),\Lambda_{1},V(2m))$, which is not
a $\PV$ for $m\geq2$. Thus the triplet
is not a $\PV$.
\end{proof}

\section{\'{E}tale PVs for reductive Lie algebras with simple Levi factors II}\label{sec6}

In this section we will discuss the \'{e}tale $\PV$s for reductive Lie algebras $\g=\g_s\oplus\fgl(1)^k$, where $\g_s$ is simple. This case is difficult because representations of $\g$ are not necessarily completely reducible and the technique we used in the previous section can be applied only when the center is 1-dimensional.

For any $\PV$ $(\g,\rho,V)$ , $V$ is completely reducible as a $\g_s$-module:
\begin{equation}\label{eq-prim-dec}
V=m_1V_1\oplus\cdots\oplus m_lV_l,
\end{equation}
where $V_i$ are mutually nonisomorphic irreducible $\g_s$-modules with multiplicity $m_i$. Thus $m_iV_i$ are $\g$-submodules and can be identified with $V_i\otimes V(m_i)$. Therefore, we may write
\begin{equation}\label{eq-decom}
(\g,\rho,V)=(\g_{s}\oplus\fgl(1)^k,\sum_{i=1}^{l}\sigma_{i}\otimes\mu_i,\bigoplus_{i=1}^lV_i\otimes V(m_i)).
\end{equation}

We have the following easy but useful observation.

\begin{lem}\label{lem-pre-sub}
Let $V_i$ be as above. Then

$(1) \ V_i\otimes V(m_i)$ is $\g$-prehomogeneous;

$(2)$ There exists a full flag $W_1\subset W_2\subset\cdots\subset W_{m_i}=V(m_i)$, which is $\fgl(1)^k$ invariant. Identifying $V(r)$ with $V(m_i)/W_{m_i-r}$, $1\leq r\leq m_i$, we have that $V_i\otimes V(r)$ is $\g_s\oplus\fgl(r)$-prehomogeneous.
\end{lem}

In particular, $V_i$ is $\g_s\oplus\fgl(1)$-prehomogeneous and must be in the list of Proposition~\ref{prop6.7}. Therefore one easily has the following.

\begin{cor}
  In the decomposition $(\ref{eq-prim-dec})$, we have $l\leq 4$. In particularly, for $\g_s=\fsp(n)$ $(n>3)$, $\operatorname{G}_2$, $\operatorname{E}_6$, $\operatorname{E}_7$, we have $l=1$.
\end{cor}

\subsection{Exceptional cases}

If $\g_s$ is exceptional, then $l=1$. Comparing the dimension of the Lie algebra and the irreducible $\PV$s in Proposition~\ref{prop6.7}, we have $m_1\leq 3$.
\begin{lem}
Let $\g_{s}$ be an exceptional simple Lie algebra. Then there are no \'{e}tale $\PV$s for $\g$.
\end{lem}
\begin{proof} From the list of irreducible $\PV$s in \cite{SK1977}, we know that there are no $\PV$s for $\g_s\oplus\fgl(3)$. 
Thus the assertion follows from Lemma \ref{lem-pre-sub}.
\end{proof}

\subsection{$\g_s=\fsp(n)$}

Note that for $n>3$ we have $l=1$ and $(\sigma_1,V_1)=(\Lambda_1,V(2n))$; for $n=2,3$, we have $l\leq 2$.

\begin{lem}\label{Sp}
Let $n>1$. Assume that $\fgl(1)^k$-representation $(\mu,V(m))$ is faithful. Then the triplet $(\fsp(n)\oplus \fgl(1)^{k},\Lambda_1\otimes\mu,V(2n)\otimes V(m))$ is prehomogeneous only if $k\leq m\leq 3$.
\end{lem}
\begin{proof}
Assume that $(\fsp(n)\oplus \fgl(1)^{k},\Lambda_1\otimes\mu,V(2n)\otimes V(m))$ is a $\PV$. Then
$(\fgl(1)^{k},\wedge^{2}\mu,V(\frac{m(m-1)}{2}))$ is a $\PV$ by Proposition~\ref{Sp(n)}, thus we have $k\geq\frac{m(m-1)}{2}$. Since the abelian subalgebras in $\fgl(m)$ has dimension at most $\left[\dfrac{m^2}{4}\right]+1$, we have $\frac{m^{2}}{4}+1\geq k\geq\frac{m(m-1)}{2}$, which implies that $m\leq 3$. For $m\leq 3$ we have $\left[\dfrac{m^2}{4}\right]+1=m$. Thus $k\leq m\leq3$.
\end{proof}

Comparing the dimensions of $V(2n)\otimes V(m)$ and Lie algebra $\g$, we have

\begin{cor}\label{cor-sp-nonpv} If $n\geq 2$, then $(\fsp(n)\oplus \fgl(1)^{k},\Lambda_1\otimes\mu,V(2n)\otimes V(m))$ is not \'{e}tale.
\end{cor}

Therefore, we just need to consider those $\fsp(n)$ which has at least 2 different irreducible $\PV$s, which implies that $n=2$ or $n=3$. For $n=2$, noticing that $(\fsp(2)\oplus\fgl(1),\Lambda_2\otimes\Box,V(5))$ is prehomogeneous since $\fsp(2)\cong\fso(5)$, we have

\begin{lem}\label{SP(2)}
The triplet
$(\fsp(2)\oplus \fgl(1)^{k},\Lambda_{2}\otimes \mu_{1}+\Lambda_{1}\otimes \mu_{2},V(5)\otimes V(m_{1})+V(4)\otimes V(m_{2}))$ is not \'{e}tale for any $\mu_{i}$ $(i=1,2)$.
\end{lem}
\begin{proof}
If $m_1>1$, then $V(5)\otimes V(2)+V(4)$ is $\fsp(2)\oplus\fgl(1)^k$-prehomogeneous. The image of $\fgl(1)^k$ is an abelian subalgebra of $\fgl(2)\oplus\fgl(1)$ and is at most $3$-dimensional. Therefore $\dim\fsp(2)+3=13<\dim(V(5)\otimes V(2)+ V(4))=14$. If $m_2>1$, then, by the same idea as in Lemma \ref{lem-pre-sub}, $V(5)+V(4)\otimes V(2)$ is prehomogeneous for $\fsp(2)\oplus \fgl(1)^{1+2}$ and hence for $\fsp(2)\oplus \fgl(2)\oplus\fgl(1)$, which contradicts the result in \cite[P.~391, Lemma 2.18]{KKIY1988}. Therefore $m_2=1$. Then the assertion follows.
\end{proof}

\begin{lem}
The triplet
$(\fsp(3)\oplus\fgl(1)^{k},\Lambda_{3}\otimes \mu_{1}+ \Lambda_{1}\otimes \mu_{2},V(14)\otimes V(m_{1})+ V(6)\otimes V(m_{2}))$ is not \'{e}tale for any $\mu_{i}$ $(i=1,2)$.
\end{lem}

\begin{proof}
It is easy to check that $(\fsp(3)\oplus\fgl(1)^{3},\Lambda_{3}\otimes\mu_{1}+\Lambda_{1}\otimes\Box,V(14)\otimes V(2)+ V(6))$ is not a $\PV$. Thus we have $m_{1}=1$ and $\mu_{1}=\Box$. By Lemma~\ref{Sp} we have $m_2\leq 3$ and $k\leq m_2+1$. The assertion follows by comparing the dimensions of the representation of the Lie algebras.
\end{proof}

\subsection{$\g_s=\fso(n)$, $n\geq 6$}

Note that only for $n=7,9,11$, $\fso(n)\oplus\fgl(1)$ has two irreducible $\PV$s: the natural one and the spin representation; for $n=8,10,12,14$, $\fso(n)$ has three 8-dimensional irreducible $\PV$s: the natural one and two half-spin representations. The two half-spins are not equivalent as representations, but equivalent as $\PV$s, and the two half-spins of $\fso(8)$ are equivalent to the natural representation as $\PV$s.

We first consider the case when $l=1$ in the decomposition (\ref{eq-decom}).

\begin{lem}\label{SO}
The triplet
$(\fso(n)\oplus \fgl(1)^{k},\Lambda_1\otimes\mu,V(n)\otimes V(m))$ $(n\geq 3,k\geq 1,m\geq 1)$
is a (faithful) $\PV$ if and only if $k=m=1$. Such $\PV$s are not \'{e}tale.
\end{lem}
\begin{proof}
The triplet is a $\PV$ if and only if $(\fgl(1)^{k},S^{2}\mu,V(\frac{m}{2}(m+1)))$ is a $\PV$ by
Proposition~\ref{SO(n)}. Then we get $k\geq\frac{m(m+1)}{2}$. Since the abelian subalgebras in $\fgl(m)$ has dimension at most $\left[\dfrac{m^2}{4}\right]+1$, we have $\frac{m^{2}}{4}+1\geq k\geq\frac{m(m+1)}{2}$. Then $k=m=1$.
\end{proof}
\begin{lem} Let $n=7,9,10,11,12,14$. The triplets $(\fso(n)\oplus \fgl(1)^{k},\spin\otimes\mu,V(2^{[\frac{n-1}{2}]})\otimes V(m))$ are not \'{e}tale for any $k,m$.
\end{lem}
\begin{proof}
By the result classification of irreducible $\PV$s of Sato and Kimura in~\cite{SK1977}, we have that, for $n=9,11,12,14$, $V(2^{[\frac{n-1}{2}]})\otimes V(2)$ is not a $\PV$ for $\fso(n)\oplus\fgl(2)$, thus it is not a $\PV$ for $\fso(n)\oplus\fgl(1)^k$. For $n=7,10$, $V(2^{[\frac{n-1}{2}]})\otimes V(3)$ are $\fso(n)\oplus\fgl(3)$-prehomogeneous, but the generic isotropy subalgebras intersects $\fso(n)$ nontrivially, which implies that $V(2^{[\frac{n-1}{2}]})\otimes V(3)$ is not $\fso(n)\oplus\fgl(1)^3$-\'{e}tale.
\end{proof}

\begin{lem}
The triplet $(\fso(n)\oplus\fgl(1)^{k},\Lambda_{1}\otimes \mu_{1}+\spin^{(\pm)}\otimes \mu_{2} ,V(n)\otimes V(m_{1})+ V(m)\otimes V(m_{2}))$ $(n=7,9,10,11,12,14)$
are not \'{e}tale for any $\mu_{i}$ $(i=1,2)$.
\end{lem}
\begin{proof}
By Lemma~\ref{SO}, we get $m_{1}=1$ and $\mu_{1}=\Box$. Now assume that $(\spin(n)\oplus\fgl(1)^{k},\Lambda_{1}\otimes \Box+\spin\otimes \mu_{2} ,V(n)\otimes V(1)+ V(2^{[\frac{n-1}{2}]})\otimes V(m_{2}))$ is a $\PV$. Then we have $m_{2}\leq3$ for $n=7,10$ and $m_{2}=1$ for $n=9,11,12,14$. One can easily check they are not \'{e}tale $\PV$.
\end{proof}

\begin{lem}\label{triSpin8}
Let
$(\fso(8)\oplus\fgl(1)^k,\Lambda_{1}\otimes \mu_{1}+\spin^+\otimes \mu_{2}+\spin^-\otimes\mu_3,V(8)\otimes V(m_{1})+ V(8)\otimes V(m_{2})+V(8)\otimes V(m_3)$
be a $\PV$. Then we have $m_1,m_2,m_3\leq 1$ and consequently the triplets cannot be \'{e}tale.
\end{lem}
\begin{proof}
Since $(\fso(8),\Lambda_{1},V(8))\simeq (\fso(8),\spin^{\pm},V(8))$. By Lemma~\ref{SO}, we have $m_{1},m_{2},m_{3}\leq 1$. And $\dim\fso(8)=28>24$, from which we deduce that the triplets cannot be \'{e}tale.
\end{proof}

\subsection{$\g_s=\fsl(n)$}

This is the most complicated case, since there are several irreducible $\PV$s for $\fgl(n)$:
$$\Lambda_1^{(*)},\Lambda_2,2\Lambda_1,\Lambda_3^{(*)}\ (n=6,7,8).$$
We will show step by step that only $\Lambda_1^{(*)}$ can occur in \'{e}tale $\PV$s.

Firstly, let us show that $\Lambda_3^{(*)}$ does not occur in \'{e}tale $\PV$s.

\begin{lem}
The triplet $(\fsl(n)\oplus \fgl(1)^{k},\Lambda_{3}\otimes\mu_{1}+\Lambda_{1}^{(\ast)}\otimes\mu_{2},V(\frac{n(n-1)(n-2)}{6})\otimes V(m_1)+ V(n)\otimes V(m_2))$ are not \'{e}tale for $n=6,7,8$.
\end{lem}
\begin{proof}
If the triplet is \'{e}tale, we can easily see that $m_1=1$ since $2\dim\Lambda_3>n^2+1$.

For $n=6$, we must have $m_2\geq 3$ and $V(20)+ V(6)\otimes V(3)$ must be $\fsl(n)\oplus\fgl(1)^{k}$-prehomogeneous. Since the generic isotropy subalgebra of $V(20)$ is $\fsl(3)\oplus\fsl(3)$, we have $V(6)\otimes V(3)$ is $\fsl(3)\oplus\fsl(3)\oplus\fgl(1)^{k-1}$-prehomogeneous. Noticing that $V(6)=V_1\oplus V_2$ $(V_1\cong V_2\cong V(3))$ as $\fsl(3)\oplus\fsl(3)$-representations, one can easily show that this representation is not prehomogeneous.

For $n=7$, the generic isotropy subalgebra of $(\fgl(7),\Lambda_3,V(35))$ is $\mathrm{G}_2$ and $\mathrm{G}_2\oplus\fgl(1)^{k-1}$ has no \'{e}tale $\PV$s.

For $n=8$, since $\dim\fsl(8)=63<\dim\Lambda_3+\dim V(8)\otimes V(2)$, we have $m_2=1$ and $k=1$. But the generic isotropy subalgebra of $(\fgl(7),\Lambda_3,V(35))$ is $\fsl(3)$, whose action on $V(8)$ is not \'{e}tale.
\end{proof}

Secondly, we show that $\Lambda_2^{(*)}$ and $2\Lambda_1^{(*)}$ do not occur in \'{e}tale $\PV$s.

\begin{lem}
  The triplet $(\fsl(n)\oplus\fgl(1)^k,2\Lambda_1\otimes\mu_1+\Lambda_2^{(*)}\otimes\mu_2,V(\frac{n(n+1)}{2})\otimes V(m_1)\oplus V(\frac{n(n-1)}{2})\otimes V(m_2))$ is not \'{e}tale for any $\mu_1,\mu_2$.
\end{lem}
\begin{proof}
  We just need to show that $(2\Lambda_1+\Lambda_2)$ is not $\fgl(n)\oplus\fgl(1)$-prehomogeneous, which follows from that the generic isotropy subalgebra of $\fgl(n)$ action on $2\Lambda_1$ is $\fso(n)$  and $\Lambda_2^{(*)}$ is not $\fso(n)\oplus\fgl(1)$-prehomogeneous since the action of $\fso(n)$ on $\Lambda_2^{(*)}$ is just the adjoint action.
\end{proof}

Therefore, if $2\Lambda_1^{(*)}$ or $\Lambda_2^{(*)}$ occurs in an \'{e}tale $\PV$ for $\fsl(n)\oplus\fgl(1)^k$, then only $\Lambda_1$ and $\Lambda_1^*$ can occur at the same time. Since $2\Lambda_1$ is a regular $\PV$ for $\fgl(n)$, it is easy to get the following

\begin{lem}\label{lem4.6}
Let
$(\fgl(n)\oplus\fgl(1)^{k},2\Lambda_{1}\otimes\mu_{1}+\Lambda_{1}\otimes\mu_{2}+\Lambda_1^*\otimes\mu_3)$
be an \'{e}tale $\PV$. Then it is equivalent to $(\fsl(2)\oplus\fgl(1)^{2},2\Lambda_1\otimes\Box+\Lambda_1\otimes\Box,V(3)+ V(2))$.
\end{lem}
\begin{proof}
It is easy to see that $\mu_1$ is 1-dimensional.
Since the generic isotropy subalgebra of $(\fgl(n),2\Lambda_{1})$ is $\fso(n)$, we must have $(\fso(n)\oplus\fgl(1)^{k-1},\Lambda_1\otimes(\mu_2+\mu_3))$ is \'{e}tale (note that $\Lambda_1\cong\Lambda_1^*$ for $\fso(n)$), which happens if and only if $n=2$ and $\mu_2+\mu_3$ is 1-dimensional, thanks to Lemma~\ref{SO}.
\end{proof}

For $n=2m$, $\Lambda_2$ is a regular $\PV$ for $\fgl(2m)$ with generic isotropy subalgebra $\fsp(m)$.

\begin{lem}\label{SL(2m)}
The triplet
$(\fsl(2m)\oplus \fgl(1)^{k},\Lambda_{2}\otimes\mu_{1}+\Lambda_{1}\otimes\mu_{2}+\Lambda_1^{\ast}\otimes\mu_3)$ ($m\geq 2$)
is not \'{e}tale for any $\mu_{1},\mu_2$ and $\mu_{3}$.
\end{lem}
\begin{proof}
If $\dim\mu_1=1$, then the generic isotropy subalgebra of the action of $\fsl(2m)\oplus\fgl(1)^k$ on $\Lambda_2$ is $\fsp(m)\oplus\fgl(1)^{k-1}$, whose action on $\Lambda_1\otimes\mu_{2}+\Lambda_1^{\ast}\otimes\mu_3$ is not \'{e}tale (see Corollary~\ref{cor-sp-nonpv}).

If $\dim\mu_1>1$, by {\cite[P.~94, Proposition 12]{SK1977}}, we see that the triplet $(\fsl(2m)\oplus \fgl(1)^{2},\Lambda_2\otimes\mu_1,V(m(2m-1))\otimes V(2))$ is not a $\PV$ for $m\geq4$. For $m=3$, the generic isotropy subalgebra of the triplet
$(\fsl(6)\oplus \fgl(2),\Lambda_2\otimes\Lambda_1,V(15)\otimes V(2))$ contained in $\fsl(6)$ (see Propositions 12 in Pages 92--94 \cite{SK1977}), which implies that $V(15)\otimes V(2)$ is not a $\PV$ for $\fsl(6)\oplus\fgl(1)^k$ since the generic isotropy subalgebra remains the same which the dimension of the Lie algebra is smaller.
For $m=2$, since $(\fsl(4)\oplus \fgl(2), \Lambda_{2}\otimes \Lambda_{1},V(6)\otimes V(2))\simeq(\fso(6)\oplus \fgl(2),\Lambda_1\otimes\Lambda_1,V(6)\otimes V(2))$. By
Lemma~\ref{SO}, we see that the triplet
$(\fsl(4)\oplus \fgl(1)^{2}, \Lambda_{2}\otimes \mu_{1},V(6)\otimes V(2))$ is not a $\PV$.
\end{proof}

For $n=2m+1$, $\Lambda_2$ is not regular for $\fgl(2m+1)$, which makes our discussion more complicated.
\begin{lem}\label{SL(2m+1)}
The triplet $(\fsl(2m+1)\oplus \fgl(1)^{k},\Lambda_{2}\otimes\mu_{1}+\Lambda_{1}\otimes\mu_{2}+
\Lambda_{1}^{\ast}\otimes\mu_{3},V(m(2m+1))\otimes V(d_1)+V(2m+1)\otimes V(d_2)+V(2m+1)\otimes V(d_3))$ cannot be an \'{e}tale $\PV$ for any $\mu_{i}$, $i=1,2,3$.
\end{lem}
\begin{proof} If $d_1\geq 3$, then, by passing to the quotient, we just need to consider the case when $d_1=3$. It is easy to see that $\Lambda_2\otimes\mu_1$ is not a $\PV$, since $\dim\mu_1(\fgl(1)^k)\leq 3$ and $\deg \Lambda_2\otimes\mu_1=3m(2m+1)>\dim\fsl(2m+1)+3$. Therefore $d_1\leq 2$.

If $d_1=2$, then we have
$$\dim\fsl(2m+1)+k=2m(2m+1)+(d_2+d_3)(2m+1),$$
i.e., $2m+k=(d_2+d_3)(2m+1)$. If $d_2+d_3>1$, we may consider the cases when $d_2=2$ and $d_3=0$ or $d_2=d_3=1$. In any case, we have $\dim\mu_2(\fgl(1)^k)+\dim\mu_3(\fgl(1)^k)\leq 2$ and can easily see that such representations are not $\PV$s. Therefore, $d_2+d_3=1$ and $k=1$. Then following \cite[Proposition~13,~P.94]{SK1977}, we identify $V(m(2m+1))\otimes V(2)$ with $\{(X_{1},X_{2})\in M(2m+1)\oplus M(2m+1)|X_{1}^{T}=-X_{1},X_{2}^{T}=-X_{2}\}$, which is a $\PV$ under the natural action of $\fgl(2m+1)$, and for
 \[X_{0}=\left\{\left(
                  \begin{array}{ccc}
                     &  & I_{m} \\
                     & 0 &  \\
                    -I_{m} &  &  \\
                  \end{array}
                \right)
,\left(
   \begin{array}{ccc}
     0 & 0 & 0 \\
     0 & 0 & -I_{m} \\
     0 & I_{m} & 0 \\
   \end{array}
 \right)
 \right\},\]
 the generic isotropy subalgebra at $X_0$ is
 $$\g_{X_0}=\left\{\left(\begin{array}{cc}a_0I_{m+1}&0\\B&-a_0I_m \end{array}\right)\left|B=\left(\begin{array}{cccc}a_1&a_2&\cdots&a_{m+1}\\a_2&a_3&\cdots&a_{m+2}\\\vdots&\vdots&&\vdots\\a_m&a_{m+1}&\cdots&a_{2m}
 \end{array}\right)\right.\right\}.$$
 Then it is easy to see that $(\g_{X_0},\Lambda_1^{(*)},V(2m+1))$ is not a $\PV$.

Now we assume that $d_1=1$. By Corollary \ref{cor4.19}, we have $d_2\leq 1$ or $d_3\leq 1$.

(i) $d_2=0$. If $(\fsl(2m+1)\oplus \fgl(1)^{k},\Lambda_2\otimes\mu_1 +\Lambda^{\ast}_1\otimes\mu_{3},V(m(2m+1))+ V(2m+1)\otimes V(d_{3}))$
is an \'{e}tale $\PV$, then by~\cite[P.~446, Proposition 1.33]{KKTI1988}, it is $\PV$-equivalent to $(\fsp(m)\oplus\fsl(2m+1)\oplus \fgl(1)^{k},\Lambda_1\otimes\Lambda_1\otimes\mu_1 +1\otimes\Lambda^{\ast}_1\otimes\mu_{3},V(2m)\otimes V(2m+1)+ V(2m+1)\otimes V(d_{3}))$. Then we have
$$m(2m+1)+(2m+1)^2-1+k=2m(2m+1)+(2m+1)d_3,$$
from which we deduce that $d_3\geq m+1$. On the other hand, by Theorem~\ref{saver}, one has that $(\fsp(m)\oplus \fgl(1)^{k},\Lambda_1\otimes\mu_{3},V(2m)\otimes V(d_{3}))$ is a $\PV$, which implies that $m(2m+1)+k\geq 2md_3$. Therefore $d_3<m+\frac{1}{2}$. A contradiction.

(ii) $d_3=0$. If $(\fsl(2m+1)\oplus \fgl(1)^{k},\Lambda_2\otimes\mu_1 +\Lambda_1\otimes\mu_{2},V(m(2m+1))+ V(2m+1)\otimes V(d_{2}))$
is an \'{e}tale $\PV$, then by~\cite[P.~446, Proposition 1.33]{KKTI1988}, it is $\PV$-equivalent to $(\fsp(m)\oplus\fsl(2m+1)\oplus \fgl(1)^{k},\Lambda_1\otimes\Lambda_1\otimes\mu_1 +1\otimes\Lambda_1\otimes\mu_{2},V(2m)\otimes V(2m+1)+ V(2m+1)\otimes V(d_{2}))$, which is castling-equivalent to
$(\fsp(m)\oplus\fsl(d_{2}-1)\oplus \fgl(1)^{k},\Lambda_1\otimes\Lambda_1\otimes\mu_1^{\ast} +1\otimes\Lambda_1\otimes\mu_{2}^{\ast},V(2m)\otimes V(d_{2}-1)+ V(d_{2}-1)\otimes V(d_{2}))$. It follows that $(\fsl(d_{2}-1)\oplus \fgl(1)^{k},\Lambda_2\otimes\mu_1^* +\Lambda_1\otimes\mu_{2}^{\ast})$ is an \'{e}tale $\PV$. A contradiction.

(iii) $d_2=1$. The generic isotropy subgroup of $(\fsl(2m+1)\oplus\fgl(1)^{2})$ acting on $\Lambda_{2}\otimes\Box+\Lambda_{1}\otimes\Box$ is $\fsp(m)\oplus\fgl(1)$,
hence the triplet $(\fsl(2m+1)\oplus \fgl(1)^{k},\Lambda_{2}\otimes\mu_{1}+\Lambda_{1}\otimes\mu_{2}+
\Lambda_{1}^{\ast}\otimes\mu_{3},V(m(2m+1))+V(2m+1)+V(2m+1)\otimes V(d_3))$ is not an \'{e}tale $\PV$ by Corollary~\ref{cor-sp-nonpv}.

(iv) $d_3=1$. If the triplet
$$(\fsl(2m+1)\oplus \fgl(1)^{k},\Lambda_{2}\otimes\mu_{1}+\Lambda_{1}\otimes\mu_{2}+
\Lambda_{1}^{\ast}\otimes\mu_{3},V(m(2m+1))+V(2m+1)\otimes V(d_2)+V(2m+1))$$ is an \'{e}tale $\PV$, then $d_2\geq m$. But we will show that the above triplet is not a $\PV$ even  for $d_2=2$.

Note that the generic isotropy subalgebra of $(\fgl(2m+1), \Lambda_{2},V(m(2m+1)))$ is
\[\fh=\left\{\left.\left(
      \begin{array}{cc}
        A & \beta \\
        0 & c \\
      \end{array}
    \right)\in\fgl(2m+1)\right|A\in\fsp(m)\right\}.\]
Consider the triplet $(\fgl(1)\oplus\fh\oplus\fgl(2),\Box\otimes\Lambda_{1}^{\ast}\otimes1+1\otimes\Lambda_{1}\otimes\Lambda_{1})$, whose isotropy subalgebra of at \[X_{1}=\left\{(1,0,\cdots,0)^{T},\left(
                              \begin{array}{ccccc}
                                1 & 0 & \cdots & 0 & 0 \\
                                0 & 0 & \cdots & 0 & 1 \\
                              \end{array}
                            \right)^{T}
\right\}\] is given by
\[\left.\left\{c\oplus\left(
           \begin{array}{ccccc}
             -c & 0 & 0 & 0 & 0 \\
             0 & D & 0 & E & 0 \\
             0 & 0 & c & 0 & 0 \\
             0 & F & 0 & -D^{T} & 0 \\
             0 & 0 & 0 & 0 & -t_{22} \\
           \end{array}
         \right)\oplus \left(
                         \begin{array}{cc}
                           c & 0 \\
                           0 & t_{22} \\
                         \end{array}
                       \right)\right|\left(\begin{array}{cc}D&E\\F&-D^T\end{array}\right)\in\fsp(m-1)\right\},
\]
which implies that $(\fsl(2m+1)\oplus \fgl(1)^{4}, \Lambda_{2}\otimes\Box+\Lambda_{1}\otimes\mu_{2}+
\Lambda_{1}^{\ast}\otimes\Box, V(m(2m+1))+V(2m+1)\otimes V(2)+V(2m+1))$ is not a $\PV$.
\end{proof}

Now that only $\Lambda_1^{(*)}$ can occur in \'{e}tale $\PV$s for $\fsl(n)\oplus\fgl(1)^k$, we just need to consider the triplet
\begin{equation}\label{eq-triplet-lambda1}
(\fsl(n)\oplus \fgl(1)^{k},\Lambda_1\otimes\mu_{1}+ \Lambda_1^{\ast}\otimes\mu_{2},V(n)\otimes V(m_{1})\oplus V(n)\otimes V(m_{2})),\ \ (n\geq 3).
\end{equation}
Since $\Lambda_1$ and $\Lambda_1^*$ are equivalent $\PV$s, we may assume that $m_1\geq m_2$. 
In the following we investigate the case of $m_2=0$.

\begin{lem}\label{lemIn}
Let the triplet $(\fsl(n)\oplus \fgl(1)^{k},\Lambda_{1}^{(\ast)}\otimes\mu,V(n)\otimes V(m))$ be an \'{e}tale $\PV$. Then $m\geq n$ and $k=n(m-n)+1$. Furthermore, identifying $V(n)\otimes V(m)$ with $M(n,m)$, we have the \'{e}tale $\PV$ has $(I_n\ 0)$ as a generic point and $\mu(\fgl(1)^k)$ consists of lower triangular matrices, up to an isomorphism of representations.
\end{lem}
\begin{proof}
The first assertion is trivial. For the second one, since $\fgl(1)^k$ is abelian, $\mu(\fgl(1)^k)$ is a $k$-dimensional abelian subalgebra of $\fgl(m)$, which may be assumed to be lower triangular. Identify $V(n)\otimes V(m)$ with $M(n,m)$ and let $X_0=(X_1\ X_2)\in M(n,m)$ be a generic point, where $X_1\in M(n,n)$. Since the orbit of $X_0$ is open and dense, we have $(X_1+t I_n\ X_2)$ is also a generic point for sufficient small $t$. Choosing $t$ such that $|X_1+tI_n|=1$, we have $Y_0=(I_n\ X_3)$ is generic for $X_3=(X_1+tI_n)^{-1}X_2$. Then there exists a lower triangular matrix $Q\in\GL(m)$ such that $Y_0 Q^T=(I_n\ 0)$. Then $(I_n\ 0)$ is a generic point for $\Lambda_1^{(*)}\otimes\mu_Q$ and $\mu_Q(A)=Q\mu(A)Q^{-1}$ is lower triangular for all $A\in\fgl(1)^k$.
\end{proof}

Now for the \'{e}tale $\PV$ in the previous Lemma, we may assume that $(I_n\ 0)$ is always a generic point and $\mu(\fgl(1)^{k})$ is a $k$-dimensional abelian subalgebra of $\fgl(m)$ consisting of lower triangular matrices.

\begin{lem}\label{lemCij} $\mu(\fgl(1)^k)=\Span\{I_m,C_{ij}\}$, where $C_{ij}=\left(\begin{array}{cc}A_{ij}&0\\E_{ij}&0\end{array}\right)$ with $A_{ij}$ lower triangular and $\tr(A_{ij})=0$. Furthermore, if $m-n>1$, then $A_{ij}=0$.
\end{lem}
\begin{proof}
Since $(I_n\ 0)$ is generic, there exists a unique element $(X,Y)\in\fsl(n)\oplus\mu(\fgl(1)^k)$ such that
$$(X,Y)\cdot(I_n\ 0)=X(I_n\ 0)+(I_n\ 0)Y^T=(I_n\ 0).$$
Set $Y=\left(\begin{array}{cc}Y_1&0\\Y_2&Y_3\end{array}\right)$. We have
\begin{equation}\label{eq-identity}
X+Y_1^{T}=I_n,\ \ \ Y_2=0.
\end{equation}
Let $V=\left\{\left.\left(\begin{array}{cc}B&0\\C&D\end{array}\right)\in\mu(\fgl(1)^k)\right|B\in\fsl(n)\right\}$. Then
$$\mu(\fgl(1)^k)=\Span\{Y\}\dot+V,$$
since $\tr(Y_1)=n$. For any nonzero $A=\left(\begin{array}{cc}B&0\\C&D\end{array}\right)\in V$, since $B\in\fsl(n)$, we have
$$(-B^T,A)\cdot(I_n\ 0)=(0\ C^T),$$
from which we deduce $C\neq 0$. Thus there exists a basis $C_{ij}$ of $V$ with the form
$$C_{ij}=\left(\begin{array}{cc}A_{ij}&0\\E_{ij}&D_{ij}\end{array}\right),$$
where $E_{ij}$ is the $(m-n)\times n$ matrix with 1 in the $(i,j)$-position and zero elsewhere.

Since $[C_{ij},Y]=0$, we have $E_{ij}Y_1=Y_3 E_{ij}$, $\forall i,j$, from which and (\ref{eq-identity}) one can easily show that $Y=I_m$.

Since $[C_{ij},C_{kl}]=0$, we have
$$E_{ij}A_{kl}+D_{ij}E_{kl}=E_{kl}A_{ij}+D_{kl}E_{ij},$$
which implies $A_{ij}=A_{kl}=0$ and $D_{ij}=D_{kl}=0$ if $m-n>1$. If $m-n=1$, the assertion is trivial since one may replace $C_{ij}$ by $C_{ij}=C_{ij}-D_{ij}Y$.
\end{proof}

Since there is always an element in $\fgl(1)^k$ acting as the identity, we may write $\g=\fsl(n)\oplus\fgl(1)^{mn-n^2+1}$ as $\fgl(n)\oplus\fgl(1)^{mn-n^2}$, where $\fgl(1)^{mn-n^2}$ is the commutative subalgebra of $\fgl(m)$ spanned by $C_{ij}$.

\begin{lem}\label{abelian2}
Let $\g=\fgl(n)\oplus \fgl(1)^{mn-n^{2}}$ be a left-symmetric algebra corresponding to an \'{e}tale prehomogeneous vector space $(\fgl(n)\oplus \fgl(1)^{mn-n^{2}},\Lambda_{1}^{(\ast)}\otimes\mu_{1},V(n)\otimes V(m))$ with $m-n>1$.
Then, as left-symmetric algebras, $\g\cong\fgl(n)\dot+\fgl(1)^{mn-n^2}$, where $\fgl(n)$ is a subalgebra with the usual matrix product and $\fgl(1)^{mn-n^{2}}$ is a commutative ideal with trivial product.
\end{lem}
\begin{proof}
Without loss of generality, we assume that $X_{0}=(I_{n},0)$ is a generic point. Then for $A,B\in \fgl(n)$, we have
\begin{align*}
A\cdot(B\cdot X_0)&=(AB)\cdot X_0;\\
A\cdot(C_{ij}\cdot X_{0})&= A(I_{n},0)\left(
                            \begin{array}{cc}
                              0 & E_{ji} \\
                              0 & 0 \\
                            \end{array}
                          \right)=(0,AE_{ji})=\left(
                                                         \begin{array}{cc}
                                                           0 & 0\\
                                                           E_{ij}A^T & 0 \\
                                                         \end{array}
                                                       \right)\cdot X_0, \\
C_{ij}\cdot(C_{st}\cdot X_{0})&= (I_{n},0)\left(
                            \begin{array}{cc}
                              0 & E_{ts} \\
                              0 & 0 \\
                            \end{array}
                          \right)\left(
                                   \begin{array}{cc}
                                     0 & E_{ji} \\
                                     0 & 0 \\
                                   \end{array}
                                 \right)
                          =0 .
\end{align*}
Denote by $*$ the left-symmetric product on $\g$. Then we have $A*B=AB$, $C_{ij}*C_{st}=0$ and $A*C_{ij}\in\fgl(1)^{mn-n^2}$, which implies that $\fgl(1)^{nm-n^{2}}$ is an ideal of $\g$ and $\fgl(n)$ is a subalgebra with the usual matrix product.
\end{proof}

\begin{rem}
This case may be illustrated as follows:
\[\fgl(n)\oplus \fgl(1)^{mn-n^{2}}\cong \fg'=\left\{\left(
                                                \begin{array}{cc}
                                                  A & 0 \\
                                                  C & 0 \\
                                                \end{array}
                                              \right)|A\in\fgl(n),C\in M(m-n,n)
\right\}. \]
For any $x,y\in\fg$, we have $x\cong\left(
                                      \begin{array}{cc}
                                        A & 0 \\
                                        C & 0 \\
                                      \end{array}
                                    \right),y\cong\left(
                                                    \begin{array}{cc}
                                                      B & 0 \\
                                                      D & 0 \\
                                                    \end{array}
                                                  \right)
$ and $x\ast y=\left(
                 \begin{array}{cc}
                   AB & 0 \\
                   CB^T+DA^T & 0 \\
                 \end{array}
               \right).
$
\end{rem}
Now we study the case of $m_2\geq 1$.

\begin{lem}\label{SL11}
Let the triplet
$$(\fsl(n)\oplus \fgl(1)^{k},\Lambda_1\otimes\mu_{1}+ \Lambda_1^{\ast}\otimes\mu_{2},V(n)\otimes V(m_{1})\oplus V(n)\otimes V(m_{2})),\ \ (n\geq 3)$$
be an \'{e}tale $\PV$ with $m_{1}\geq m_{2}\geq1$. Then we have

$(1) \ m_{2}=1;$

$(2) \ m_1=n$ and $k=n+1;$

$(3)$ we may write $\fsl(n)\oplus\fgl(1)^{n+1}$ as $\fgl(n)\oplus\fgl(1)^n$, where
$$\mu_{1}(\fgl(1)^{n})=\Span\left\{I_{n},\left(
                                               \begin{array}{cc}
                                                 0 & 0 \\
                                                 E_{i1} & C_{i} \\
                                               \end{array}
                                             \right),i=2,\ldots,n
\right\},$$
where $C_{2},\ldots,C_{n}\in\fgl(n-1)$ are lower triangular matrices.
\end{lem}
\begin{proof}
(1) follows from Corollary \ref{cor4.19}.

(2) Since $k\geq 2$, it is easy to see that $m_1\geq n$ and $V(n)\otimes V(m_1)$ is $\fsl(n)\oplus\fgl(1)^k$-prehomogeneous. It is easy to see that any generic isotropy subalgebra for $V(n)\otimes V(m_1)$ intersects $\fsl(n)$ trivially. We may choose subalgebra $\fgl(1)^r$ of $\fgl(1)^k$ such that $V(n)\otimes V(m_1)$ is \'{e}tale for $\fsl(n)\oplus\fgl(1)^r$. If $r>1$, then, by Lemma~\ref{lemCij}, it is easy to check that $\mu_1(\fgl(1)^r)$ is a maximal abelian subalgebra which doesnot contain in any other abelian subalgebra and $r=(m_1-n)n+1$. Noticing that $k=(m_1+1-n)n+1$ and $\dim\mu_2(\fgl(1)^k)\leq 1$, we have $\dim\mu_1(\fgl(1)^k)\geq k-1=(m_1+1-n)n$. Since $\mu_1(\fgl(1)^k)\subseteq\mu_1(\fgl(1)^r)$, it is a contradiction.

(3) Similar to the proof of Lemma~\ref{lemCij}.
\end{proof}

Combining all the results we obtained above, we have

\begin{thm}\label{s+c^{k}}
Let $(\g_{s}\oplus \fgl(1)^{k},\sum_{i}\sigma_{i}\otimes\mu_{i},\sum_{i}V_{i}\otimes V'_{i})$ be an \'{e}tale $\PV$, where $\g_{s}$ is simple. Then it is equivalent to one of the following triplets:

$(1) \ (\fsl(n)\oplus\fgl(1),\Lambda_{1}\otimes\mu_{1},V(n)\otimes V(n))\ (\tr(\mu_{1}(\fgl(1)))\not\equiv 0);$

$(2) \ (\fsl(2)\oplus\fgl(1),3\Lambda_{1}\otimes\Box,V(4));$

$(3) \ (\fsl(n)\oplus\fgl(1)^{n+1},\Lambda_{1}^{(\ast)}\otimes\mu_{1},V(n)\otimes V(n+1)),$ \ (for $\mu_{1}$, see Lemma~\ref{lemCij})$;$

$(4) \ (\fgl(n)\oplus\fgl(1)^{mn-n^2},\Lambda_{1}^{(\ast)}\otimes\mu_{1},V(n)\otimes V(m)),(m-n>1)$ \ (for $\mu_{1},$ see Lemma~\ref{abelian2})$;$

$(5) \ (\fsl(n)\oplus\fgl(1)^{n+1},\Lambda_{1}\otimes\mu_{1}+\Lambda_{1}^{\ast}\otimes\Box,V(n)+V(n)\otimes V(n)),$ \ (for $\mu_{1},$ see Lemma~\ref{SL11})$;$

$(6) \ (\fsl(2)\oplus\fgl(1)^{2},2\Lambda_1\otimes\Box+\Lambda_1\otimes\Box,V(3)+ V(2)).$
\end{thm}

\section*{Acknowledgements} This work was partially supported by National Natural Science Foundation of China (11571182). The authors would like to thank Professor J. A. Wolf at Berkeley for helpful discussion.

\end{document}